\documentclass[12pt,draft]{amsart}
\usepackage[all]{xy}
\usepackage{amsfonts}
\usepackage{amssymb}
\usepackage{color}

\newtheorem{teo}{Theorem}[section]
\newtheorem{lem}[teo]{Lemma}
\newtheorem{prop}[teo]{Proposition}

\newtheorem{cor}[teo]{Corollary}
\newtheorem{conj}[teo]{Conjecture}
\theoremstyle{definition}
\newtheorem{dfn}[teo]{Definition}
\newtheorem{rk}[teo]{Remark}
\newtheorem{ex}[teo]{Example}
\def\<{\langle}
\def\>{\rangle}
\def\ss{\subset}

\def\a{\alpha}
\def\b{\beta}

\def\e{\varepsilon}
\def\l{{\lambda}}
\def\r{\rho}

\def\t{\tau}

\def\w{\omega}
\def\f{{\varphi}}

\def\G{{\Gamma}}
\def\F{{\Phi}}

\def\C{{\mathbb C}}

\def\Z{{\mathbb Z}}

\def\A{{\mathcal A}}
\def\K{{\mathcal K}}

\def\End{\mathop{\rm End}\nolimits}
\def\Ker{\mathop{\rm Ker}\nolimits}

\def\Coker{\operatorname{Coker}}
\def\Id{\operatorname{Id}}

\def\Prim{\operatorname{Prim}}

\def\Tr{\operatorname{Tr}}
\def\supp{\operatorname{supp}}

\def\T{{\mathbb T}}
\def\1{\mathbf 1}

\def\ola#1{\stackrel{#1}{\longrightarrow}}

\newcommand{\ov}[1]{\overline{#1}}
\newcommand{\til}[1]{\widetilde{#1}}
\newcommand{\wh}[1]{\widehat{#1}}

\newcommand{\Mat}[4]{\left( \begin{array}{cc}
                            #1 & #2 \\
                            #3 & #4
                      \end{array} \right)}
\newcommand{\Matt}[9]{\left( \begin{array}{ccc}
                            #1 & #2 & #3 \\
                            #4 & #5 & #6 \\
                            #7 & #8 & #9
                      \end{array} \right)}

\def\Fix{\operatorname{Fix}}
\def\Rep{\operatorname{Rep}}

\def\N{{\mathbb N}}
\newcommand\hk[1]{\stackrel\frown{#1}}
\def\cd{\operatorname{cd}}

\begin{document}

\title[Reidemeister classes and twisted Burnside theorem]
{Geometry of Reidemeister classes
and twisted Burnside theorem}

\author{Alexander Fel'shtyn}
\address{Instytut Matematyki, Uniwersytet Szczecinski,
ul. Wielkopolska 15, 70-451 Szczecin, Poland and Department of Mathematics,
Boise State University,
1910 University Drive, Boise, Idaho, 83725-155, USA}
\email{felshtyn@diamond.boisestate.edu, felshtyn@mpim-bonn.mpg.de}
\author{Evgenij Troitsky}
\thanks{The second author is partially supported by
RFFI Grant  05-01-00923
and Grant ``Universities of Russia''.}
\address{Dept. of Mech. and Math., Moscow State University,
119992 GSP-2  Moscow, Russia}
\email{troitsky@mech.math.msu.su}
\urladdr{
http://mech.math.msu.su/\~{}troitsky}


\begin{abstract}
The purpose of the present mostly expository
paper (based mainly on \cite{FelTro,FelTroVer,ncrmkwb,polyc,FelGon})
is to present the current state of the following conjecture of
A.~Fel'shtyn and R.~Hill \cite{FelHill}, which is a generalization
of the classical Burnside theorem.

Let $G$ be a countable discrete group, $\phi$ one of its  automorphisms,
$R(\phi)$ the number of $\phi$-conjugacy (or twisted conjugacy)
classes, and
$S(\phi)=\# \Fix (\wh\phi)$ the
number of $\phi$-invariant equivalence classes of irreducible
unitary representations. If one of $R(\phi)$ and
$S(\phi)$ is finite, then it is equal to the other.

This conjecture plays a  important role in the theory of
twisted conjugacy classes (see \cite{Jiang}, \cite{FelshB})
and has very important  consequences
in Dynamics, while its proof needs rather sophisticated  results from
Functional and Non-commutative Harmonic Analysis.

First we prove this conjecture for finitely generated groups of type I
and discuss its applications.

After that we discuss an important example of an automorphism of
a type II${}_{1}$ group which disproves the original formulation
of the conjecture.

Then we prove a version of the conjecture for a wide class of groups,
including almost polycyclic groups (in particular, finitely generated groups
of polynomial growth). In this formulation the role of
an appropriate dual object plays the finite-dimensional part of the unitary
dual. Some counter-examples are discussed.

Then we begin a  discussion of
the general case (which also needs new definition of the dual object)
and prove the weak twisted Burnside theorem for general countable
discrete groups. For this purpose we prove a non-commutative version
of Riesz-Markov-Kakutani representation theorem.

Finally we explain why the Reidemeister numbers are always infinite
for Baumslag-Solitar groups.
\end{abstract}

\dedicatory{Dedicated to the memory of Professor Yuri Solovyov}

\maketitle

\tableofcontents

\section{Introduction and formulation of results}

\begin{dfn}
Let $G$ be a countable discrete group and $\phi: G\rightarrow G$ an
endomorphism.
Two elements $x,x'\in G$ are said to be
 $\phi$-{\em conjugate} or {\em twisted conjugate,}
iff there exists $g \in G$ with
$$
x'=g  x   \phi(g^{-1}).
$$
We shall write $\{x\}_\phi$ for the $\phi$-{\em conjugacy} or
{\em twisted conjugacy} class
 of the element $x\in G$.
The number of $\phi$-conjugacy classes is called the {\em Reidemeister number}
of an  endomorphism $\phi$ and is  denoted by $R(\phi)$.
If $\phi$ is the identity map then the $\phi$-conjugacy classes are the usual
conjugacy classes in the group $G$.
\end{dfn}

If $G$ is a finite group, then the classical Burnside theorem (see e.g.
\cite[p.~140]{Kirillov})
says that the number of
classes of irreducible representations is equal to the number of conjugacy
classes of elements of $G$.  Let $\wh G$ be the {\em unitary dual} of $G$,
i.e. the set of equivalence classes of unitary irreducible
representations of $G$.

\begin{rk}
If $\phi: G\to G$ is an epimorphism, it induces a map $\wh\phi:\wh G\to\wh G$,
$\wh\phi (\r)=\r\circ\phi$
(because a representation is irreducible if and only if the  scalar operators in the space of
representation are the only ones which commute with all operators of the
representation). This is not the case for a general endomorphism $\phi$,
because $\r \phi$ can be reducible for an irreducible representation $\r$,
and  $\wh\phi$ can be defined only as a multi-valued map.
But
nevertheless we can define the set of fixed points $\Fix \wh\phi$
of $\wh\phi$ on $\wh G$.
\end{rk}

Therefore, by the Burnside's theorem, if $\phi$ is the identity automorphism
of any finite group $G$, then we have
 $R(\phi)=\#\Fix(\wh\phi)$.

To formulate the main theorem of the first part of the paper
for the case
of a general endomorphism we first need
an appropriate definition of the $\Fix(\wh\phi)$.

\begin{dfn}
Let $\Rep(G)$ be the space of equivalence classes of
finite dimensional unitary representations of $G$.
Then the corresponding map $\wh\phi_R:\Rep(G)\to \Rep(G)$
is defined in the same way as above: $\wh\phi_R(\r)=\r\circ\phi$.

Let us denote by $\Fix(\wh\phi)$ the set of points $\r\in\wh G\ss
\Rep(G)$ such that $\wh\phi_R (\r)=\r$.
\end{dfn}

\begin{teo}
\label{teo:mainth1} Let $G$ be a finitely generated discrete group
of type {\rm I}, $\phi$ one of its endomorphism, $R(\phi)$ the
number of $\phi$-conjugacy classes, and $S(\phi)=\# \Fix
(\wh\phi)$ the number of $\wh\phi$-invariant equivalence classes
of irreducible unitary representations. If one of $R(\phi)$ and
$S(\phi)$ is finite, then it is equal to the other.
\end{teo}

Let $\mu(d)$, $d\in\N$, be the {\em M\"obius function},
i.e.
$$
\mu(d) =
\left\{
\begin{array}{ll}
1 & {\rm if}\ d=1,  \\
(-1)^k & {\rm if}\ d\ {\rm is\ a\ product\ of}\ k\ {\rm distinct\ primes,}\\
0 & {\rm if}\ d\ {\rm is\ not\ square-free.}
\end{array}
\right.
$$

\begin{teo}[Congruences for the Reidemeister numbers]\label{teo:mainth3}
Let $\phi:G$ $\to G$ be an endomorphism of a countable discrete group $G$
such that
all numbers $R(\phi^n)$ are finite and let $H$ be a subgroup
 of $G$ with the properties
$$
  \phi(H) \subset H,
$$
$$
  \forall x\in G \; \exists n\in \N \hbox{ such that } \phi^n(x)\in H.
$$
If the pair  $(H,\phi^n)$ satisfies the conditions of
Theorem~{\rm~\ref{teo:mainth1}}
for any $n\in\N$,
then one has for all $n$,
 $$
 \sum_{d\mid n} \mu(d)\cdot R(\phi^{n/d}) \equiv 0 \mod n.
 $$
\end{teo}

These theorems were proved previously in a   special case of
Abelian finitely generated plus finite group \cite{FelHill,FelBanach}.

For groups of type II${}_1$ the situation is much more complicated.
We discuss in detail the case of a semi-direct product of the action
of $\Z$ on $\Z\oplus \Z$ by a hyperbolic automorphism with finite
Reidemeister number (four to be precise) and the number of fixed points
of $\wh\phi$ on $\wh G$ equal or greater than five~\cite{FelTroVer}.
The origin of this
phenomenon lies in bad separation properties of $\wh G$ for general
discrete groups. A more deep study leads to the following general theorem.

 \begin{teo}[{weak twisted Burnside theorem, \cite{ncrmkwb}}]\label{teo:weakburn}
The number $R_*(\phi)$ of Reidemeister classes related to
twisted invariant functions on $G$ from the Fourier-Stieltjes
algebra $B(G)$ is  equal to the number $S_*(\phi)$ of generalized
fixed points of $\wh\phi$ on the Glimm spectrum of $G$, i.~e. on
the complete regularization of $\wh G$, if one of $R_*(\phi)$ and
$S_*(\phi)$ is finite.
 \end{teo}

The argument goes along the following line.
The well-known Riesz(-Markov-Kakutani) theorem identifies
the space of linear functionals on algebra $A=C(X)$ and the
space of regular measures on $X$. To prove the weak twisted Burnside
theorem we first obtain a generalization of this theorem to the
case of non-commutative $C^*$-algebra $A$ via Dauns-Hofmann
sectional representation theorem. The corresponding measures
on Glimm spectrum are functional-valued.
In extreme situation this theorem is tautological, but for
group $C^*$-algebras of discrete groups in many cases one
obtains some new tool for counting twisted conjugacy classes.

\medskip

Keeping in mind that for hyperbolic groups $R(\phi)$ is always infinite
while in the "opposite" case the twisted Burnside theorem is proved, we
can formulate the following conjecture, which is in fact a program of
further actions.

\begin{conj}
There exists a class of groups $\mathcal G$ such that
\begin{itemize}
\item for any group
$G\not\in \mathcal G$ and any automorphism $\phi:G\to G$ the
Reidemeister number $R(\phi)$ is always infinite,

\item for any group $G\in \mathcal G$ there exists a subset
of ideals $\mathcal{M}\ss \Prim\, C^*(G)$ such that
its points are separated and $R(\phi)$ coincides with the number
of fixed points of $\wh\phi$ on $\mathcal M$ supposing one of these numbers
to be finite.
\end{itemize}
\end{conj}

One of candidates for $\mathcal M$ is the set of maximal ideals (let
us remind that the Glimm spectrum is in fact the space of maximal ideals
of the center of $C^*(G)$, cf. \cite{DaunsHofmann,ncrmkwb}).
The strategy of proof will be based on the weak theorem
\ref{teo:weakburn}.

We also consider a formulation of the main conjecture with
counting only finite-di\-men\-si\-o\-n\-al fixed points on the unitary
dual. In \cite{polyc} we prove this version for a very large class
of groups.

\begin{teo}\label{teo:polyc}
For a wide class of groups, which includes almost polycyclic groups
$($in particular, finitely generated groups of polynomial growth$)$
the Reidemeister number $R(\phi)$ coincides with the number $S_f(\phi)$ of
finite-dimensional $\wh\phi$-fixed points on the unitary
dual, if $R(\phi)$ is finite.
\footnote{After the acceptance of the present paper for publication, we obtained
a more strong version of this theorem using another method.
Namely, the end of the theorem sounds as
``if one of these numbers is finite''. For this purpose we introduce
the notion of $\phi$-conjugacy separable group (for any two $\phi$-conjugacy classes
there exists a $\phi$-respecting homomorphism of $G$ onto a finite group such
that the images of these classes do not intersect). A sufficient condition for
this property is conjugacy separability of the semidirect product of $G$ and $\Z$
by $\phi$. By the famous result of Remeslenikov and Formanek this is
the case for any almost polycyclic group.}
\end{teo}

This version of the main conjecture (with the consideration
of only finite-dimensional representations)
has some counter-examples. One of them, coming
from D.~Osin's group \cite{Osin} we discuss after the mentioned
theorem in Section \ref{sec:polyc}. Also we propose there some
possible formulation of the conjecture appropriate for residually
finite groups.

Let us remark, that in some cases the weak twisted Burnside theorem
easily implies the twisted Burnside theorem (in particular, in the form with
finite-dimensional representations). For example, we will show
directly that $R(\phi)=R_*(\phi)$ in Abelian case. On the other
hand, the unitary dual coincides with the Glimm spectrum.
Slightly more complicated argument is valid for some
more general groups covered by Theorem \ref{teo:polyc}, in particular,
the Heisenberg group. More precisely, it is possible to extract from
the first part of the proof of Theorem \ref{teo:polyc} (see below
Section \ref{sec:polyc}) that $R(\phi)=R_*(\phi)$.
Moreover, characteristic functions of Reidemeister classes are
related to finite dimensional representations, which are maximal and
Hausdorff separated. Hence, the finite-dimensional fixed points
form a part of generalized fixed points on Glimm spectrum and
$S_*(\phi)\ge S_f(\phi)$. In fact, we have an equality here, because
if there exists a functional coming from some other generalized fixed
point on Glimm spectrum, it cannot be a coefficient of a finite-dimensional
representation because Glimm spectrum is Hausdorff separated while all
finite-dimensional (generalized) fixed points are already counted in $S_f(\phi)$.
The details concerning the relation between separateness and linear
independence are contained in Section \ref{sec:polyc}.

\medskip
The interest in twisted conjugacy relations has its origins, in particular,
in the Nielsen-Reidemeister fixed point theory (see, e.g. \cite{Jiang,FelshB}),
in Selberg theory (see, eg. \cite{Shokra,Arthur}),
and  Algebraic Geometry (see, e.g. \cite{Groth}).

The congruences give some necessary conditions for the realization problem
for Reidemeister numbers in topological dynamics.
The relations with Selberg theory will be presented in a forthcoming paper.

Let us remark that it is known that the Reidemeister number of an
endomorphism of a finitely generated Abelian group is
finite iff $1$ is not in the
spectrum of the restriction of this endomorphism to the free part of the group
(see, e.g. \cite{Jiang}).  The Reidemeister number
is infinite for any automorphism of a  non-elementary
Gromov hyperbolic group
\cite{FelPOMI}.

\medskip
To make the presentation more detailed and transparent we start
from a new approach (E.T.) for Abelian (Section
\ref{sec:abelcase}) and compact (Section \ref{sec:compcase})
groups. Only after that we develop this approach and prove the
main theorem for finitely generated groups of type I \cite{FelTro}
in Section \ref{sec:typeI}. A discussion of some examples leading
to conjectures is the subject of Section \ref{sec:exampdisc}.  In Section
\ref{sec:vershik} we present the mentioned important example for
type II${}_1$ groups \cite{FelTroVer}. After that we pass to the
demonstration of the weak twisted Burnside theorem in Section
\ref{sec:invfunsme}. For this purpose we present in Section
\ref{sec:funstmeas} a non-commutative version of
Riesz-Markov-Kakutani theorem \cite{ncrmkwb} (some necessary
information on operator fields is collected in Section
\ref{sec:algfiel}).
In Section \ref{sec:polyc} we discuss almost polycyclic
groups and related matter.
Section \ref{sec:BaumSolit} is devoted to
the discussion of Baumslag-Solitar groups \cite{FelGon}.

\medskip\noindent
{\bf Acknowledgement.} We would like to thank the Max Planck
Institute for Mathematics in Bonn for its kind support and
hospitality while the most part of this work has been completed.
We are also indebted to MPI and organizers of the
Workshops on Noncommutative Geometry and Number Theory I, II (Bonn,
August 2003 and June 2004)
where the results of this paper were presented.

The first author thanks the Research Institute in Mathematical Sciences
in Kyoto and Texas A \& M University
for the possibility of the present research during his visit there.

The authors are grateful to
V.~Balantsev,
G.~Banaszak,
M.~B.~Bekka,
B.~Bowditch,
R.~Grigorchuk,
R.~Hill,
V.~Kaimanovich,
V.~Ma\-nui\-lov,
A.~Mish\-che\-n\-ko,
A.~Rosenberg,
M.~Sapir,
A.~Shtern,
L.~Vainerman,
A.~Vershik
for helpful discussions, and to the referee for valuable suggestions.

\section{Abelian case}\label{sec:abelcase}

Let $\phi$ be an automorphism of an Abelian group $G$.

\begin{lem}\label{lem:charabelclassred}
The twisted conjugacy class $H$ of $e$ is a subgroup.
The other ones  are cosets $gH$.
\end{lem}

\begin{proof}
The first statement follows from the equalities
$$
h\phi(h^{-1}) g\phi(g^{-1})=gh \phi((gh)^{-1},\quad
(h\phi(h^{-1}))^{-1}=\phi(h) h^{-1}=h^{-1} \phi(h).
$$
For the second statement suppose $a\sim b$, i.e. $b=ha\phi(h^{-1})$. Then
$$
gb=gha\phi(h^{-1})=h(ga)\phi(h^{-1}), \qquad gb\sim ga.
$$
\end{proof}

\begin{lem}
Suppose, $u_1,u_2\in G$, $\chi_H$ is the
characteristic function of $H$ as a set. Then
$$
\chi_H(u_1 u_2^{-1})=\left\{
\begin{array}{ll}
1,& \mbox{ if } u_1,u_2 \mbox{ are in one coset },\\
0, & \mbox{ otherwise }.
\end{array}
\right.
$$
\end{lem}

\begin{proof}
Suppose, $u_1\in g_1 H$, $u_2\in g_2 H$, hence, $u_1=g_1 h_1$, $u_2=g_2 h_2$.
Then
$$
u_1 u_2^{-1}=g_1 h_1 h_2^{-1} g_2^{-1} \in g_1 g_2^{-1} H.
$$
Thus, $\chi_H(u_1 u_2^{-1})=1$ if and only if
$g_1 g_2^{-1}\in  H$ and $u_1$ and $u_2$
are in the same class. Otherwise it is 0.
\end{proof}

The following Lemma is well known.

\begin{lem}
For any subgroup $H$ the function $\chi_H$ is of positive type.
\end{lem}

\begin{proof}
Let us take arbitrary elements $u_1,u_2,\dots, u_n$ of $G$.
Let us reenumerate them
in such a way that some first are in $g_1 H$, the next ones are  in $g_2 H$,
and so on, till
$g_m H$, where $g_j H$ are different cosets. By the previous Lemma the matrix
$\|p_{it}\|:=\|\chi_H(u_i u_t^{-1})\| $ is block-diagonal with
square blocks formed by
units. These blocks, and consequently  the whole  matrix are  positively
semi-defined.
\end{proof}

\begin{lem}
In the Abelian case characteristic functions of twisted conjugacy classes
belong to
the Fourier-Stieltjes algebra $B(G)=(C^*(G))^*$.
\end{lem}

\begin{proof}
By Lemma \ref{lem:charabelclassred}
in this case the characteristic functions of twisted conjugacy
classes are the shifts
of the characteristic function of the class $H$ of $e$.
Hence, by Corollary (2.19) of \cite{Eym}, these characteristic functions are in
$B(G)$.
\end{proof}

Let us remark that there exists a natural isomorphism (Fourier transform)
$$
u\mapsto\wh u,\qquad C^*(G)=C_r^*(G)\cong C(\wh G),\qquad
\wh g(\rho):=\rho(g),
$$
(this is a  number because irreducible representations of an Abelian group
are 1-dimen\-sional). In fact, it is better to look (for what follows) at an  algebra $C(\wh G)$ as an algebra of continuous sections  of a bundle of 1-dimensional matrix algebras.
over $\wh G$.

Our characteristic functions, being in $B(G)=(C^*(G))^*$ in this case,
are mapped to the functionals on $C(\wh G)$  which, by the
Riesz-Markov-Kakutani theorem,
are measures on $\wh G$. Which of these measures  are invariant under
the induced (twisted) action of $G$ ? Let us remark,
that an invariant non-trivial
functional gives rise to at least one invariant space -- its kernel.

Let us remark, that convolution under the Fourier transform becomes
point-wise multiplication. More precisely, the twisted action, for example,
is defined as
$$
g[f](\rho)=\rho(g)f(\rho)\rho(\phi(g^{-1})),\qquad \rho\in\wh G,\quad
g\in G,\quad f\in C(\wh G).
$$

There are 2 possibilities  for the twisted action of $G$ on the representation
 algebra $A_\rho \cong \C$: 1) the linear span of the orbit of $1 \in A_\rho$
is equal to all $A_\rho$, 2) and  the opposite case (the action is trivial).

The second case means that the space of intertwining operators between
 $A_\rho$ and $A_{\wh \phi \rho}$ equals $\C$, and  $\rho$ is a fixed point of
the action $\wh \phi:\wh G\to \wh G$. In the first case this is the opposite
situation.

If we have a finite number of such fixed
points, then the space of twisted invariant measures is just the space of
measures
concentrated in these points. Indeed, let us describe the action of $G$
on measures in
more detail.

\begin{lem}
For any  Borel set $E$ one has
$g[\mu](E)=\int_E g[1]\,d\mu$.
\end{lem}

\begin{proof}
The restriction of measure to any Borel set commutes with the action of  $G$,
since the last is point wise on $C(\wh G)$.
For any  Borel set $E$ one has
$$
g[\mu](E)= \int_E 1\, dg[\mu]= \int_E g[1] \,d\mu.
$$
\end{proof}

Hence, if $\mu$ is twisted invariant, then for any  Borel set $E$ and any
$g\in G$ one has
$$
\int_E (1-g[1])\,d\mu =0.
$$

\begin{lem}
Suppose, $f\in C(X)$, where $X$ is a compact Hausdorff space, and $\mu$
is a regular Borel measure on $X$, i.e. a functional on $C(X)$.
Suppose, for any
Borel set $E\subset X$ one has $\int_E f\, d\mu=0$.
Then $\mu (h)=0$ for any $h\in C(X)$
such that $f(x)=0$ implies $h(x)=0$. I.e. $\mu$ is concentrated off the
interior
of $\rm supp\,f$.
\end{lem}

\begin{proof}
Since the functions of the form $fh$ are dense in the space of the
referred to above  $h$'s,
it is sufficient to verify the statement for $fh$. Let us choose an
arbitrary $\e>0$
and a simple function $h'=\sum\limits_{i=1}^n a_i \chi_{E_i}$
such that $|\mu(fh')-\mu(fh)|<\e$.
Then
$$
\mu(fh')=\sum_{i=1}^n \int_{E_i} a_i f \,d\mu =\sum_{i=1}^n a_i
\int_{E_i}f \,d\mu =0.
$$
Since $\e$ is an arbitrary one, we are done.
\end{proof}

Applying this lemma to a twisted invariant measure $\mu$ and $f=1-g[1]$
we obtain that $\mu$ is concentrated at our finite number of fixed
points of $\wh\phi$, because outside of them $f\ne 0$.

If we have an infinite number of fixed points, then the space is
infinite-dimensional
(we have an infinite number of measures
concentrated in finite number of points, each time different)
and Reidemeister number is infinite as well. So, we are done.

\section{Compact case}\label{sec:compcase}

Let $G$ be a compact Hausdorff
group, hence $\wh G$ is a discrete space. Then
 $C^*(G)=\oplus M_i$, where $M_i$ are the matrix algebras of
irreducible representations.
The infinite sum is in the following sense:
$$
C^*(G)=\{ f_i\}, i\in \{ 1,2,3,...\}=\hat G, f_i\in M_i,
\| f_i \| \to 0 (i\to \infty).
$$
When $G$ is finite and $\wh G$ is finite this is exactly Peter-Weyl theorem.

A characteristic function of a twisted class
is a functional on $C^*(G)$. For a  finite group it is evident,   for a
general compact group  it is necessary to verify only the measurability of
the twisted class
with the respect to Haar measure, i.e. that twisted class is Borel.
For a compact $G$,
the twisted conjugacy classes being orbits of twisted action are compact and
hence closed.

Under the identification it passes to a sequence
$\{ \f_i \} $, where $\f_i$ is a functional on $M_i$ (the properties
of convergence can be formulated, but they play no role at the moment).
The conditions of invariance are the following: for each $\r_i \in\wh G$
one has $g[\f_i]=\f_i$, i.e. for any $a\in M_i$ and any $g\in G$ one has
 $\f_i(\r_i(g) a \r_i(\phi(g^{-1})))=\f_i(a)$.

Let us recall  the following well-known fact.

\begin{lem}\label{lem:funcnamat}
Each functional on matrix algebra has form
$a\mapsto \Tr(ab)$ for a fixed matrix $b$.
\end{lem}

\begin{proof} One has $\dim (M(n,\C))'=\dim (M(n,\C))=n\times n$
and looking
at matrices as at operators in $V$, $\dim V=n$, with base $e_i$,
one can remark that
functionals $a\mapsto \<a e_i, e_j\>$, $i,j=1,\dots, n$,
are linearly  independent. Hence,
any functional takes form
$$
a\mapsto \sum_{i,j} b^i_j \<a e_i, e_j\>=
\sum_{i,j} b^i_j a^j _i = \Tr(ba),\qquad
b:=\|b^i_j\|.
$$
\end{proof}

Now we can study invariant functionals:
$$
\Tr(b\r_i(g) a \r_i(\phi(g^{-1})))=\Tr(ba),\qquad \forall\,a,g,
$$
$$
\Tr((b-\r_i(\phi(g^{-1})) b \r_i(g))a)=0,\qquad \forall\,a,g,
$$
hence,
$$
b-\r_i(\phi(g^{-1})) b \r_i(g)=0,\qquad \forall\,g.
$$
Since $\r_i$ is irreducible, the dimension of the space of such $b$ is
1 if $\r_i$ is a fixed
point of $\wh\phi$ and 0 in the opposite case by the Schur lemma.
So, we are done.

\begin{rk}
In fact we are only interested in finite discrete case.
Indeed, for a compact $G$,
the twisted conjugacy classes being orbits of twisted action are
compact and hence closed.
If there is a finite number of them, then they are open as well.
Hence, the situation is more or less
reduced to a discrete group: quotient by the component of unity.
\end{rk}

\section{Extensions and Reidemeister classes}\label{sec:extens}

Let us denote by $\t_g:G\to G$ the automorphism $\t_g(\til g)=g\til g\,g^{-1}$
for $g\in G$. Its restriction on a normal subgroup we will denote by $\t_g$
as well.

\begin{lem}\label{lem:redklassed}
$\{g\}_\phi k=\{g\,k\}_{\t_{k^{-1}}\circ\phi}$.
\end{lem}

\begin{proof}
Let $g'=f\,g\,\phi(f^{-1})$ be $\phi$-conjugate to $g$. Then
$$
g'\,k=f\,g\,\phi(f^{-1})\,k=f\,g\,k\,k^{-1}\,\phi(f^{-1})\,k
=f\,(g\,k)\,(\t_{k^{-1}}\circ\phi)(f^{-1}).
$$
Conversely, if $g'$ is $\t_{k^{-1}}\circ\phi$-conjugate to $g$, then
$$
g'\,k^{-1}=
f\,g\,(\t_{k^{-1}}\circ\phi)(f^{-1})k^{-1}=
f\,g\,k^{-1}\,\phi(f^{-1}).
$$
Hence a shift maps $\phi$-conjugacy classes onto classes related to
another automorphism.
\end{proof}

\begin{cor}\label{cor:rphiequaltaurphi}
$R(\phi)=R(\t_g \circ \phi)$.
\end{cor}

Consider a group extension respecting homomorphism $\phi$:
\begin{equation}\label{eq:extens}
 \xymatrix{
0\ar[r]&
H \ar[r]^i \ar[d]_{\phi'}&  G\ar[r]^p \ar[d]^{\phi} & G/H \ar[d]^{\ov{\phi}}
\ar[r]&0\\
0\ar[r]&H\ar[r]^i & G\ar[r]^p &G/H\ar[r]& 0,}
\end{equation}
where $H$ is a normal subgroup of $G$.
The argument below, especially
related the role of fixed points, has a partial intersection with
\cite{gowon,goncalves}.

First of all let us notice that the Reidemeister classes of $\phi$ in $G$
are mapped epimorphically on classes of $\ov\phi$ in $G/H$. Indeed,
\begin{equation}\label{eq:epiofclassforexs}
p(\til g) p(g) \ov\phi(p(\til g^{-1}))= p (\til g g \phi(\til g^{-1}).
\end{equation}
Suppose, $R(\phi)<\infty$. Then the previous remark implies
$R(\ov\phi)<\infty$. Consider a class $D=\{h\}_{\t_g\phi'}$, where
$\t_g(h):=g h g^{-1},$ $g\in G$, $h\in H$. The corresponding equivalence
relation is
\begin{equation}\label{eq:klasstaug}
h\sim \til h h g \phi'(\til h^{-1}) g^{-1}.
\end{equation}
Since $H$ is normal,
the automorphism $\t_g:H\to H$ is well defined.
We will denote
by $D$ the image $iD$ as well. By (\ref{eq:klasstaug})
the shift $Dg$ is a subset of $Hg$ is characterized by
\begin{equation}\label{eq:klasstaugsh}
h g\sim \til h (h g) \phi'(\til h^{-1}).
\end{equation}
Hence it is a subset of $\{hg\}_\phi\cap Hg$ and the partition
$Hg=\cup (\{h\}_{\t_g \phi' }) g$ is a subpartition of
$Hg=\cup ( Hg\cap \{hg\}_\phi).$

We need the following statements.

\begin{lem}\label{lem:ozenkacherezperiod}
Suppose, the extension {\rm (\ref{eq:extens})} satisfies
the following conditions:
\begin{enumerate}
    \item $\#\Fix\ov{\phi}=k <\infty$,
    \item $R(\phi)<\infty$.
\end{enumerate}
Then
\begin{equation}\label{eq:ozenreidcherpernonabel}
    R(\phi')\le k\cdot (R(\phi)-R(\ov\phi)+1).
\end{equation}

If $G/H$   is abelian, let $g_i$ be some elements with
$p(g_i)$ being representatives of all
different $\ov\phi$-conjugacy classes, $i=1,\dots,R(\ov\phi)$.
Then
\begin{equation}\label{eq:ozenreidcherperabel}
   \sum_{i=1}^{  R(\ov{\phi})} R(\t_{g_i}\phi')\le k\cdot R(\phi).
\end{equation}
\end{lem}

\begin{proof}
Consider classes $\{z\}_{\phi'}$, $z\in G$, i.e. the classes of
relation $z\sim hz\phi'(h^{-1})$, $h\in H$. The group $G$ acts on
them by $z\mapsto gz\phi(g^{-1})$. Indeed,
\begin{multline*}
    g[\til h h \phi(\til h ^{-1})]\phi(g^{-1})=
(g \til h g^{-1}) (g h \phi(g^{-1}))(\phi(g) \phi(\til h ^{-1})\phi(g^{-1}))\\
    =
(g \til h g^{-1}) (g h \phi(g^{-1})) \phi(g \til hg^{-1})\in
\{g h \phi(g^{-1})\}_{\phi'},
\end{multline*}
because $H$ is normal and $g \til h g^{-1}\in H$. Due to invertibility,
this action of $G$ transposes classes $\{z\}_{\phi'}$ inside one class
$\{g\}_\phi$. Hence, the number $d$ of classes $\{h\}_{\phi'}$
inside $\{h\}_{\phi}\cap H$ does not exceed the number of $g\in G$
such that $p(g)\ov\phi(p(g^{-1}))=\ov e$.
Since two elements $g$ and $gh$
in one $H$-coset induce the same permutation of classes $\{h\}_{\phi'}$,
the mentioned number $d$ does not exceed the number of $z\in G/H$
such that $z\ov\phi(z^{-1})=\ov e$, i.e. $d\le k$. This implies
(\ref{eq:ozenreidcherpernonabel}).

Now we discuss $\phi$-classes over $\ov\phi$-classes other than
$\{\ov e\}_{\ov\phi}$ for an abelian $G/H$. An estimation analogous to
the above one leads to the number of $z\in G/H$
such that $z z_0\ov\phi(z^{-1})=z_0$ for some fixed $z_0$. But for
an Abelian $G/H$ they form the same group $\Fix(\ov\phi)$. This
together with the description (\ref{eq:klasstaugsh})
of shifts of $D$ at the beginning of the present Section implies
(\ref{eq:ozenreidcherperabel}).
\end{proof}

\begin{lem}\label{lem:fixedofextens}
Suppose, in the extension {\rm (\ref{eq:extens})} the group $H$ is
abelian. Then $\#\Fix(\phi)\le \#\Fix(\phi')\cdot \#\Fix(\ov\phi)$.
\end{lem}

\begin{proof}
Let $s:G/H\to G$ be a section of $p$. If $s(z)h$ is a fixed
point of $\phi$ then
\begin{equation}\label{eq:sotziphiotz}
    (s(z))^{-1}\phi(s(z))=h\phi'(h^{-1}).
\end{equation}
Hence, $z\in \Fix(\ov\phi)$ and left hand side takes $k:=\#\Fix(\ov\phi)$
values $h_1,\dots,h_k$.
Let us estimate the number of $s(z)h$ for a fixed $z$ such that
$(s(z))^{-1}\phi(s(z))=h_i$. These $h$ have to satisfy
(\ref{eq:sotziphiotz}).
Since $H$ is abelian, if one has
$$
h_i=h\phi'(h^{-1})=\til h\phi'(\til h^{-1}),
$$
then $h^{-1}\til h\in \Fix(\phi')$ and we are done.
\end{proof}

\begin{teo}\label{teo:fixandreidforabel}
Let $A$ be a finitely generated Abelian group, $\psi:A\to A$ its
automorphism with $R(\psi)<\infty$. Then $\#\Fix(\psi)<\infty$.

Moreover, $R(\psi)\ge \#Fix(\psi)$.
\end{teo}

\begin{proof}
Let $T$ be the torsion subgroup. It is finite and characteristic.
We obtain the extension $T\to A \to A/T$ respecting $\phi$.
Since $A/T\cong \Z^k$, $\Fix(\ov\psi:A/T\to A/T)=\ov e$, by
\cite{Jiang},\cite[Sect. 2.3.1]{FelshB}. Hence, by Lemma
\ref{lem:fixedofextens}, $\#\Fix(\psi)\le \#\Fix(\psi')$,
$\psi':T\to T$. For any finite abelian group $T$ one
clearly has $\#\Fix(\psi')=R(\psi')$ by Theorem \ref{teo:reidandcokerabel}
(cf. \cite[p.~7]{FelshB}). Finally,
$R(\psi')\le R(\psi)$ by (\ref{eq:ozenreidcherperabel}).
\end{proof}

\begin{lem}\label{lem:ozenfixp}
Suppose, $|G/H|=N<\infty$. Then
$R(\t_g\phi')\le N R(\phi)$. More precisely, the mentioned
subpartition is not more than in $N$ parts.
\end{lem}

\begin{proof}
Consider the following action of $G$ on itself: $x\mapsto gx\phi(g^{-1}).$
Then its orbits are exactly classes $\{x\}_\phi$. Moreover it maps classes
(\ref{eq:klasstaugsh}) onto each other. Indeed,
$$
\til g \til h (h g) \phi'(\til h^{-1}) \phi(\til g^{-1})=
 {\wh h}  \til g (h g)  \phi(\til g^{-1})
\phi'( {\wh h}^{-1})
$$
using normality of the $H$. This map is invertible
($\til g \leftrightarrow \til g^{-1}$), hence bijection.
Moreover, $\til g$ and $\til g \wh h$, for any $\wh h\in H,$
act in the same way. Or in the other words, $H$ is in the stabilizer
of this permutation of classes (\ref{eq:klasstaugsh}). Hence,
the cardinality of any orbit $\le N$.
\end{proof}

Hence, for any finite $G/H$ the number of classes of the form
(\ref{eq:klasstaugsh}) is finite: it is $\le N  R(\phi)$.

\begin{lem}
Suppose,   $H$ satisfies the following
property: for any automorphism of $H$ with finite Reidemeister
number the characteristic functions of Reidemeister classes of $\phi$
are linear combinations
of matrix elements of some finite number of irreducible
finite dimensional  representations of $H$.
Then the characteristic functions of classes
{\rm (\ref{eq:klasstaugsh})}
are linear combinations
of matrix elements of some finite number of irreducible
 finite dimensional  representations of $G$.
\end{lem}

\begin{proof}
Let $\r_1,\r_2,\dots,\r_k$ be the above  irreducible
representations of $H$, $\r$ its direct sum acting on $V$, and
$\pi$ the regular
(finite dimensional) representation of $G/H$.
Let $\r^I_1,\dots,\r^I_k, \r^I$ be the corresponding
induced representations of $G$. Let us remind that in this
simple situation the representation $\r^I$ is defined as
a representation of $G$ in the space
$l_2(G/H,V)\cong \oplus_1^{|G/H|} V$ defined by the formula
$$
[\r^I(g) f](x)=\r(s(x)g(s(x g))^{-1})f(x g),\qquad f\in
l_2(G/H,V), \quad x\in G/H,
$$
for some fixed section $s:G/H\to G$ of the canonical projection
$G\to G/H$.

Let the characteristic function of
$D$ be represented under the form $\chi_D(h)=\<\r(h)\xi,\eta\>$.
Let $\xi^I\in L^2(G/H,V)$ be defined by the formulas
$\xi^I(\ov e)=\xi \in V$, $\xi^I(\ov g)=0$ if $\ov g\ne \ov e$.
Define similarly $\eta^I$. Then for $h\in i(H)$ we have
$$
\r^I(h) \xi^I (\ov g)=\r(s(\ov g) h s(\ov g h)^{-1})\xi (\ov g h)=
\r(h s(\ov g) s(\ov g)^{-1})\xi (\ov g )=
\begin{cases}
\r(h)\xi, & \mbox{ if } \ov g = \ov e, \\
0,  & \mbox{ otherwise. }
\end{cases}
$$
Hence,
$\<\r^I(h) \xi^I,\eta^I\>|_{i(H)}$ is the characteristic function of
$i(D)$. Let $u,v\in L^2(G/H)$ be such vectors that $\<\pi(\ov g)u,v\>$
is the characteristic function of $\ov e$. Then
$$
\<(\r^I\otimes \pi)(\xi^I\otimes u,\eta^I\otimes v)\>
$$
is the characteristic function of  $i(D)$. Other characteristic
functions of classes
{\rm (\ref{eq:klasstaugsh})}
are shifts of this one. Hence, they are matrix elements of the
representation $\r^I\otimes \pi$. It is finite dimensional. Hence it can
be decomposed in a finite direct sum of irreducible representations.
\end{proof}

\begin{cor}[of previous two lemmata]\label{lem:osnoskl}
Under the assumptions  of the previous lemma, the
characteristic functions of Reidemeister classes of $\phi$
are linear combinations
of matrix elements of some finite number of irreducible
finite dimensional representations of $G$.
\end{cor}

\section{The case of groups of type I}\label{sec:typeI}

\begin{teo}\label{teo:proptypeI}
Let $G$ be a discrete group of type {\rm I}. Then
\begin{itemize}
\item \cite[3.1.4, 4.1.11]{DixmierEng}
The dual space $\wh G$ is a $T_1$-topological space.

\item \cite{ThomaInvent}
Any irreducible representation of $G$ is finite-dimensional.

\end{itemize}
\end{teo}

\begin{rk}\label{rk:typeIareHausd}
In fact a discrete group  $G$ is of type I if and only if it has a
normal, Abelian subgroup $M$ of finite index. The dimension of any
irreducible representation of $G$ is at most $[G: M]$
\cite{ThomaInvent}.
\end{rk}

Suppose $R=R(\phi)<\infty$, and let $F\ss L^\infty(G)$ be the
$R$-dimensional space of all twisted-invariant functionals on
$L^1(G)$. Let $K\ss L^1(G)$ be the intersection of kernels of
functionals from $F$. Then $K$ is a linear subspace of $L^1(G)$
of codimension $R$. For each $\rho\in\wh G$ let us denote by
$K_\r$ the image $\rho(K)$. This is a subspace of a
(finite-dimensional) full matrix algebra. Let $\cd_\r$ be
its codimension.

Let us introduce the following set
$$
\wh G_F=\{\r\in\wh G \: |\: \cd_\r\ne 0\}.
$$

\begin{lem}
One has $\cd_\r\ne 0$ if and only if  $\r$ is a fixed point
of $\wh\phi$.
In this case $\cd_\r=1$.
\end{lem}

\begin{proof}
Suppose, $\cd_\r\ne 0$ and let us choose a functional $\f_\r$
on the (finite-dimensional full matrix) algebra $\r(L^1(G))$
such that $K_\r\ss \Ker \f_\r$. Then for the corresponding functional
$\f^*_\r=\f_\r\circ\r$ on $L^1(G)$ one has $K\ss \Ker \f^*_\r$.
Hence, $\f^*_\r\in F$ and is twisted-invariant, as well as $\f_\r$.
Then we argue as in the case of compact group (text after Lemma
\ref{lem:funcnamat}).

Conversely, if $\r$ is a fixed point of $\wh \phi$,
it gives rise to a (unique up to scaling) non-trivial
twisted-invariant functional $\f_\r$. Let $x=\r(a)$ be any
element in $\r(L^1(G))$
such that $\f_\r(x)\ne 0$. Then $x\not\in K_\r$, because
$\f^*_\r(a)=\f_\r(x)\ne 0$, while $\f^*_\r$ is a twisted-invariant
functional on $L^1(G)$. So, $\cd_\r\ne 0$.

The uniqueness (up to scaling) of the intertwining operator implies
the uniqueness of the corresponding twisted-invariant functional.
Hence, $\cd_\r=1$.
\end{proof}

Hence,
\begin{equation}\label{eq:gfifix}
\wh G_F = \Fix (\wh \phi).
\end{equation}
From the property $\cd_\r=1$
one obtains for this (unique up to scaling) functional $\f_\r$:
\begin{equation}\label{eq:kerphirho}
\Ker\f_\r=K_\r.
\end{equation}

\begin{lem}\label{lem:RiGF}
$R=\# \wh G_F$, in particular, the set $\wh G_F$ is finite.
\end{lem}

\begin{proof}
First of all we remark that since $G$ is finitely generated
almost Abelian
(cf. Remark \ref{rk:typeIareHausd}) there is a normal Abelian subgroup
$H$ of finite index invariant under $\phi$. Hence
we can apply Lemma \ref{lem:osnoskl} to $G$, $H$, $\phi$.
So there is  a finite collection of irreducible
representations of $G$ such that
any twisted-invariant functional is a linear combination of
matrix elements of them, i.e. linear combination of functionals on
them. If one of them gives a non-trivial contribution, it has
to be a twisted-invariant functional on the corresponding matrix
algebra. Hence, by the argument above, these representations belong
to $\wh G_F$, and the appropriate functional is unique up to scaling.
Hence, $R\le S$.

Then we use $T_1$-separation property. More precisely,
suppose some points $\r_1,\dots,\r_s$ belong to $\wh G_F$.
Let us choose some twisted-invariant functionals
$\f_i=\f_{\r_i}$ corresponding to these points as it was
described
(i.e. choose some scaling).
Assume that $\|\f_i\|=1$,
$\f_i(x_i)=1$, $x_i\in\r_i(L^1(G))$.
If we can find $a_i\in L^1(G)$ such that
$\f_i(\r_i(a_i))=\f^*_i(a_i)$
is sufficiently large
and $\r_j(a_i)$, $i\ne j$, are sufficiently
small (in fact it is sufficient $\r_j(a_i)$ to be close enough
to $K_j:=K_{\r_j}$),
then $\f^*_j(a_i)$ are small for $i\ne j$,
and $\f^*_i$ are linear independent and hence,
$s<R$.  This would imply $S:=\# \wh G_F \le R$ is finite.
Hence, $R=S$.

So, the problem is reduced to the search of the above  $a_i$.
Let $d=\max\limits_{i=1,\dots,s} \dim \r_i$.
For each $i$ let $c_i:=\|b_i\|$, where $x_i$ is the unitary equivalence
of $\r_i$ and $\wh\phi \r_i$ and $x_i=\r_i(b_i)$.

Let $c:=\max\limits_{i=1,\dots,s} c_i$
and $\e:=\frac 1{2\cdot s^2\cdot d \cdot c}$.

One can find a positive element $a'_i\in L^1(G)$ such that
$\|\r_i(a'_i)\|\ge 1$ and $\|\r_j(a'_i)\|<\e $ for $j\ne i$.
Indeed, $\r_i$ can be separated from one point, and hence
from the finite number of points:
$\r_j$, $j\ne i$. Hence, one can find
an element $v_i$ such that $\|\r_i(v_i)\|>1$, $\|\r_j(v_i)\|<1$
for $j\ne i$ \cite[Lemma 3.3.3]{DixmierEng}. The same is true for
the positive element $u_i=v_i^*v_i$.
(Due to density we do not distinguish elements of $L^1$ and $C^*$).
Now for a sufficiently large
$n$ the element $a'_i:=(u_i)^n$ has the desired properties.

Let us take $a_i:=a'_i b_i^*$. Then
\begin{multline}\label{eq:ozendiag}
\f^*_i(a_i)=\Tr(x_i \r_i(a_i))=\Tr (x_i \r_i(a'_i) \r_i(b_i)^*)=
\Tr (x_i \r_i(a'_i) x_i^*)\\
=\Tr(x_i \r_i(a'_i) (x_i)^{-1})
=\Tr( \r_i(a'_i))\ge \frac 1{\dim\r_i}
\ge \frac 1d.
\end{multline}
For $j\ne i$ one has
\begin{equation}\label{eq:ozenvnediag}
\|\f^*_j(a_i)\|=\|\f_j(\r_j(a'_i b_i^*))\|\le c_i \cdot \e.
\end{equation}
Then the $s\times s$ matrix $\F=\f^*_j(a_i)$ can be decomposed into the
sum of the diagonal matrix $\Delta$ and off-diagonal $\Sigma$. By
(\ref{eq:ozendiag}) one has
$\Delta\ge \frac 1d$. By (\ref{eq:ozenvnediag}) one has
$$
\|\Sigma\|\le s^2 \cdot c_i \cdot \e \le s^2 \cdot c \cdot
\frac 1{2\cdot s^2\cdot d \cdot c}= \frac 1{2d}.
$$
Hence, $\F$ is non-degenerate and we are done.
\end{proof}

Lemma \ref{lem:RiGF} together with (\ref{eq:gfifix})
completes the proof of
Theorem \ref{teo:mainth1} for automorphisms.

We need the
following additional observations
for the proof of Theorem \ref{teo:mainth1} for
a general endomorphism
(in which (3) is false for
infinite-dimensional representations).

\begin{lem}\label{lem:endomfin}
\begin{enumerate}
\item If $\phi$ is an epimorphism, then $\wh G$ is $\wh\phi_R$-invariant.

\item For any $\phi$ the set $\Rep(G)\setminus \wh G$ is $\wh\phi_R$-invariant.

\item The dimension of the space of intertwining operators between
$\r\in\wh G$ and $\wh\phi_R(\r)$ is equal to $1$ if and only if
$\r\in \Fix(\wh\phi)$. Otherwise it is $0$.

\end{enumerate}
\end{lem}

\begin{proof}
(1) and (2):
This follows from the characterization of irreducible representation
as that one for which the centralizer of $\r(G)$ consists exactly of
scalar operators.

(3) Let us decompose $\wh\phi_R(\r)$ into irreducible ones. Since
$\dim H_\r=\dim H_{\wh\phi(\r)}$ one has only 2 possibilities:
$\r$ does not appear  in $\wh\phi(\r)$ and the intertwining number is $0$,
otherwise $\wh\phi_R(\r)$ is equivalent to $\r$. In this case
$\r\in \Fix(\wh\phi)$.
\end{proof}

The proof of Theorem \ref{teo:mainth1} can be now repeated
for the general endomorphism
with the new definition of $\Fix(\wh\phi)$. The item (3) supplies us
with the necessary property.

\section{Examples and their discussion}\label{sec:exampdisc}

The natural candidate for the dual object to be used  instead of
$\wh G$ in the case when the different notions of the dual do not
coincide (i.e. for groups more general than type I one groups) is
the so-called quasi-dual $\hk G$, i.e. the set of
quasi-equivalence classes of factor-representations (see, e.g.
\cite{DixmierEng}). This is a usual object when we need a sort of
canonical decomposition for regular representation or group
$C^*$-algebra. More precisely, one needs the support $\hk G_p $
of the Plancherel measure.

Unfortunately the following example shows that this is not the case.

\begin{ex} Let $G$ be a non-elementary Gromov hyperbolic group.
As it was shown by Fel'shtyn \cite{FelPOMI} with the
help of geometrical methods,
for any automorphism $\phi$ of $G$ the Reidemeister number $R(\phi)$
 is infinite. In particular this is true for free group in two generators
$F_2$. But the support $(\hk {F_2})_p $ consists of one point
(i.e. regular representation is factorial).
\end{ex}

The next hope was to exclude from this dual object the $II_1$-points
assuming that they always give rise to an infinite number of twisted
invariant functionals.
But this is also wrong:
\begin{ex}\label{ex:FGW} (an idea of Fel'shtyn realized in \cite{gowon})
Let $G=(Z\oplus Z)\rtimes_\theta  Z $ be the semi-direct product by a
hyperbolic action $\theta(1)=\Mat 2111$. Let
$\phi $
be an automorphism of $G$ whose restriction to  $Z$ is $-id$ and restriction
to  $Z\oplus Z$ is $\Mat 01{-1}0$.
Then $R(\phi)=4$, while the space $\hk G_p $
consists of a single $\rm II_1$-point once again (cf. \cite[p.~94]{ConnesNCG}).
\end{ex}

These examples show that powerful methods of the decomposition
theory do not work for more general classes of groups.

On the other hand Example \ref{ex:FGW}
disproves the old conjecture of Fel'shtyn
and Hill \cite{FelHill} who supposed that the Reidemeister numbers of
an injective endomorphism
for groups of exponential growth are always infinite.
More precisely, this group is amenable and of exponential growth.
Also, one has the following example
\begin{ex}
In \cite{Osin} D.~Osin has constructed an infinite finitely generated
group, which is not amenable, contains the free group in two generators,
and has two (ordinary) conjugacy classes.
\end{ex}
The role of this example will be clarified in Section \ref{sec:polyc}.

In this relation  the following example (to be discussed in Section
\ref{sec:BaumSolit})
seems to be interesting.
\begin{ex}\cite{FelGon}
For amenable and non-amenable Baumslag-Solitar groups Reidemeister
numbers are always infinite.
\end{ex}

For Example \ref{ex:FGW} recently we have found 4 fixed points of
$\wh\phi$ being finite dimensional irreducible
representations. They give rise to
4 linear independent twisted invariant functionals. These
functionals  can also be
obtained from the regular factorial representation. There  also exist
fixed points (at least one)
that are  infinite dimensional irreducible
representations. The corresponding
functionals are evidently linear dependent with the first 4. This
example will be presented in detail in Section \ref{sec:vershik}.

 \section{An example for type II${}_1$ groups}\label{sec:vershik}

In the present section based on \cite{FelTroVer}
it is shown that the twisted Burnside theorem (or Fel'shtyn-Hill
conjecture) in the original formulation with $\wh G$ as a dual
object is not true for non-type I groups.
More precisely, an example of a group and its automorphism is
constructed such that the number of fixed irreducible
representations is greater than the Reidemeister number. But the
number of fixed finite-dimensional representations
(i.e. the number of invariant finite-dimensional characters)
in this example coincides with the Reidemeister number.
The directions for search of an appropriate formulation are indicated
(another definition of the dual object). Some advances in this
direction will be made in the next sections.

\medskip
Let $G$ be a semidirect product of $\Z^2$ and $\Z$ by Anosov automorphism
$\a$
with the matrix $A=\Mat 2111$. It consists by the definition of triples
$((m,k),n)$ of integers with the following multiplication low:
$$
((m,k),n)*((m',k'),n')=((m,k)+\a^n(m',k'),n+n').
$$
In particular,
$$
((m,k),0)*((0,0),n)=((m,k),n).
$$
The inverse of $((m,k),n)$ is $(-\a^{-n}(m,k),-n)$. Indeed,
$$
((m,k),n)*(-\a^{-n}(m,k),-n)=((m,k)-\a^n \a^{-n}(m,k),n-n)=((0,0),0).
$$

The group $G$ is a solvable (hence, amenable) group which is not
of type I. Its regular representation is factorial. The irreducible
representations can be obtained from ergodic orbits of the action of
$\a$ on the torus $\T^2$ which is dual for the normal subgroup $\Z^2$
using appropriate cocycles \cite[Sec.~II.4]{ConnesNCG},
\cite[Ch.~17, \S~1]{baron}.

Let us define an automorphism $\phi:G\to G$ by
$$
\phi((m,k),n)=((k,-m),-n),
$$
i.e. the action on $\Z^2$ is defined by automorphism $\mu$ with the matrix
$M=\Mat 01{-1}0$, and on $\Z$ by $n\mapsto -n$. The map $\phi$ is clearly
a bijection,
\begin{multline*}
\phi(((m,k),n)*((m',k'),n')=\phi((m,k)+\a^n(m',k'),n+n'))\\
=
\phi((k,-m)+\mu\a^n(m',k'),-n-n'),
\end{multline*}
\begin{multline*}
\phi((m,k),n)*\phi(((m',k'),n')=((k,-m),-n)*((k',-m'),-n') \\
=((k,-m)+\a^{-n}(k',-m'),-n-n').
\end{multline*}
Hence, to prove that $\phi$ is an automorphism, we need $\mu\a^n=\a^{-n}\mu$.
This follows from $\mu\a=\a^{-1}\mu$. The further results in this direction
can be found in \cite{gowon}.

Let us find the Reidemeister classes of $\phi$, i.e. the classes of the
equivalence relation $h\sim gh\phi(g^{-1})$. For $h=((m,k),n)$ and
$g=((x,y),z)$ the right hand side of the relation takes the following
form:
\begin{multline}\label{eq:formconj}
((x,y)+ \a^z(m,k),z+n)*(-\mu\a^{-z}(x,y),z) \\
=((x,y)+ \a^z(m,k)-\a^{z+n}\mu\a^{-z}(x,y),2z+n) \\
=(\a^z\{(m,k) + (\Id - \a^n\mu) \a^{-z}(x,y)\}, 2z+n).
\end{multline}
Let us call {\em level} $n$ (of $G$) the coset $L_n$ of $\Z^2\ss G$
of all elements of the form $((m,k),n)$. Let us first take an element
$((m,k),0)$ from the level $0$ and describe elements from the same
level, being equivalent to it. By (\ref{eq:formconj}) in this case
$z=n=0$ and they have the form
$$
((m,k) + (\Id - \mu)(x,y), 0)=((m + (x-y),k + (x+y)),0),
$$
where $\Id - \mu$ has the matrix $\Mat 1{-1}11$. Hence, the level $0$
has intersections with 2 Reidemeister classes, say, $B_1$ and $B_2$.
The first intersection $B_1\cap L_0$
is formed by elements $((u,v),0)$ with even $u+v$,
and $B_2\cap L_0$ --- with odd $u+v$. The elements from the other
levels, which are equivalent to $((m,k),0)$, have the form
\begin{equation}\label{eq:even1}
(\a^z\{(m,k) + (\Id - \mu) \a^{-z}(x,y)\}, 2z).
\end{equation}
This means that $B_1$ and $B_2$ enter only even levels.
Also, since $\a$ is an automorphism, we can rewrite (\ref{eq:even1})
as
\begin{equation}\label{eq:even2}
(\a^z\{(m,k) + (\Id - \mu)(u,v)\}, 2z).
\end{equation}
with arbitrary integers $u$ and $v$. This means, that the intersections
$B_i\cap L_{2z}$ have the form $\a^z(B_i)$, $i=1,2$. In particular,
the other Reidemeister classes do not enter even levels.

In a similar way, the elements of $L_1$ equivalent to $((m,k),1)$
have the form
$$
((m,k) + (\Id - \a\mu)(x,y), 1)=((m + (2x-2y),k + x),1).
$$
This means, that $L_1$ enters 2 classes: the intersection with $B_3$ is formed by elements
with even first coordinate, and with $B_4$ --- with the odd one.
The elements from the other
levels, which are equivalent to $((m,k),1)$, have the form
\begin{equation}\label{eq:odd1}
(\a^z\{(m,k) + (\Id - \a\mu) \a^{-z}(x,y)\}, 2z+1).
\end{equation}
Since $\a$ is an automorphism, we can rewrite (\ref{eq:odd1})
as
\begin{equation}\label{eq:odd2}
(\a^z\{(m,k) + (\Id - \a \mu)(u,v)\}, 2z+1).
\end{equation}
with arbitrary integers $u$ and $v$. This means, that the intersections
$B_i\cap L_{2z+1}$ have the form $\a^z(B_i)$, $i=3,4$. In particular,
these four classes cover $G$ and $R(\phi)=4$.

To obtain a complete description of $B_i$ let us remark that directly
from the definition of $\a$
$$
\a(x,y)=(2x+y,x+y)
$$
one has the following properties:

\begin{itemize}
\item $\a$ maps the set of elements with an even
(resp., odd) sum of coordinates onto the set of elements with an even
(resp., odd) second coordinate,

\item $\a$ maps the set of elements   with an even
(resp., odd) second coordinate onto the set of elements with an even
(resp., odd) first coordinate,

\item $\a$ maps the set of elements with an even
(resp., odd) first coordinate
 onto the set of elements with an even
(resp., odd) sum of coordinates.
\end{itemize}

Hence, the elements $((m,k),n)$ in intersections $B_i\cap L_j$
are of the form

\begin{center}
\begin{tabular}{|c|c|c|c|c|}
\hline
$i$  & 1 & 2 & 3 & 4 \\
\hline
\hline
$j\equiv 0 \mod 6$ & $m+k$ is even &$m+k$ is odd & $\emptyset$ & $\emptyset$ \\
\hline
$j\equiv 1 \mod 6$ &$\emptyset$ &$\emptyset$ & $m$ is even & $m$ is odd \\
\hline
$j\equiv 2 \mod 6$ & $k$ is even & $k$ is odd & $\emptyset$ & $\emptyset$ \\
\hline
$j\equiv 3 \mod 6$ &$\emptyset$ &$\emptyset$  & $m+k$ is even & $m+k$ is odd \\
\hline
$j\equiv 4 \mod 6$ & $m$ is even & $m$ is odd & $\emptyset$ & $\emptyset$ \\
\hline
$j\equiv 5 \mod 6$ &$\emptyset$ &$\emptyset$ & $k$ is even & $k$ is odd \\
\hline
\end{tabular}
\end{center}

Now we want to study the fixed points of the homeomorphism
$\wh\phi:\wh G\to \wh G$, $[\r]\mapsto [\r\phi]$ of the unitary dual.
Let us start from the finite-dimensional representations. As it was
shown in \cite{FelTro} (see also Section \ref{sec:compcase} and after)
in this case there exists exactly one twisted-invariant
functional on $L^1(G)$, or $\phi$-central $L^\infty$ function, coming
from a twisted-invariant functional on $\r(L^1(G))\cong M(\dim \r,\C)$
(up to scaling), namely
\begin{equation}\label{eq:trsg}
\f_\r:g\mapsto \Tr(S\r(g)),
\end{equation}
where $S$ is the intertwining operator between $\r$ and $\r\phi$.

First, we have to find $\mu$-invariant finite $\a$-orbits on $\T^2$.
One can notice that
$$
\det(A^n-M)=\det A^n + 1=2
$$
for any $n$. Hence, the mentioned orbits are formed by points with
coordinates $0$ and $1/2$. We have $2$ orbits: one of them consists of
$1$ point $(0,0)$ and gives rise to $1$-dimensional trivial representation
$\r_1$,
and the other consists of $A_1=(0,1/2)$, $A_2=(1/2,0)$ and
$A_3=(1/2,1/2)$ and
gives rise to a $3$-dimensional (irreducible) representation $\r_2$.
Also, one has the following $1$-dimensional representation $\pi$:
$$
\pi((m,k),2n)=1,\qquad \pi((m,k),2n+1)=-1.
$$
So, we have $4$ representations
$$
\r_1,\quad \r_2, \quad \pi,\quad \r_2 \otimes \pi.
$$
We claim that they give rise via (\ref{eq:trsg}) to $4$ linear independent
twisted-invariant functionals. In particular, there is no more
finite-dimensional fixed points of $\wh\phi$.

Clearly,
\begin{equation}\label{eq:trsg1}
\f_{\r_1}\equiv 1, \quad \f_\pi =\left\{
\begin{array}{rl}
1, & \mbox{ on } \cup_n L_{2n}=B_1\cup B_2,\\
-1,& \mbox{ on } \cup_nL_{2n+1}=B_3\cup B_4,
\end{array}
\right.
\quad
\f_{\r_2\otimes \pi}=\f_{\r_2}\cdot \f_\pi.
\end{equation}
Let us find $\f_{\r_2}$. In the space $L^2(A_1,A_2,A_3)$ we take the base
$\e_1$, $\e_2$, $\e_3$ of characteristic functions in these points.
One has
$$
\a(\e_1)=\e_3,\qquad \a(\e_3)=\e_2,\qquad \a(\e_2)=\e_1,
$$
$$
\mu(\e_1)=\e_2,\qquad \mu(\e_2)=\e_1, \qquad \mu(\e_3)=\e_3.
$$
The representation (see \cite[Ch.~17, \S~1]{baron}) is defined by:
$$
\r_2(m,k,0)(\e_i)=\chi_{A_i}(m,k)\cdot \e_i,\qquad
\r_2(0,0,n)(\e_i)= \a^{-n} (\e_i)=\e_{i+n \mod 3},
$$
$$
\r_2(m,k,0)(\e_1)=e^{2\pi i(0\cdot m+ 1/2\cdot k)}\cdot \e_1
= e^{\pi i k} \e_1,
$$
$$
\r_2(m,k,0)(\e_2)=e^{2\pi i(1/2 \cdot m+ 0\cdot k)}\cdot \e_2
= e^{\pi i m} \e_2,
$$
$$
\r_2(m,k,0)(\e_3)=e^{2\pi i(1/2 \cdot m+ 1/2\cdot k)}\cdot \e_3
= e^{\pi i (m+k)} \e_3.
$$
The representation $\wh\phi\r_2$ is defined by
$$
\wh\phi\r_2(0,0,n)(\e_i)=\r_2(0,0,-n)(\e_i)
=\e_{i-n \mod 3},
$$
$$
\wh\phi\r_2(m,k,0)(\e_1)=\r_2(k,-m,0)(\e_1)
= e^{- \pi i m} \e_1= e^{\pi i m} \e_1,
$$
$$
\wh\phi\r_2(m,k,0)(\e_2)=\r_2(k,-m,0)(\e_2)
= e^{\pi i k} \e_2,
$$
$$
\wh\phi
\r_2(m,k,0)(\e_3)=\r_2(k,-m,0)(\e_3)
= e^{\pi i (k-m)} \e_3= e^{\pi i (m+k)} \e_3.
$$
The intertwining operator is induced by $\phi$ and has
the matrix $S=\Matt 010100001$. Hence,
$$
\f_{\r_2}(m,k,n)=\Tr(S\r_2(m,k,0)\r_2(0,0,n))
$$
$$
=
\Tr\left[ \Matt 010100001
\Matt {e^{\pi i k}}000{e^{\pi i m}}000{e^{\pi i(k+m)}}
\Matt 001100010 ^n
\right]
$$
$$
=
\Tr\left[
\Matt 0{e^{\pi i k}}0{e^{\pi i k}}0000{e^{\pi i(k+m)}}
{\Matt 001100010}^n
\right].
$$
For $n\equiv 0\mod 3$
$$
\f_{\r_2}(m,k,n)=\Tr\left[
\Matt 0{e^{\pi i m}}0{e^{\pi i k}}0000{e^{\pi i(k+m)}}
\right]=e^{\pi i(k+m)}=\left\{
\begin{array}{rl}
1,& \mbox{ if } m+k \mbox{ is even }\\
-1,& \mbox{ if } m+k \mbox{ is odd }
\end{array}
\right. ,
$$
for $n\equiv 1\mod 3$
$$
\f_{\r_2}(m,k,n)=\Tr\left[
\Matt 0{e^{\pi i m}}0{e^{\pi i k}}0000{e^{\pi i(k+m)}}
\Matt 001100010
\right]
$$
$$
=
\Tr\Matt {e^{\pi i m}}0000{e^{\pi i k}}0{e^{\pi i(k+m)}}0
=e^{\pi i m}=\left\{
\begin{array}{rl}
1,& \mbox{ if } m \mbox{ is even }\\
-1,& \mbox{ if } m \mbox{ is odd }
\end{array}
\right. ,
$$
for $n\equiv 2\mod 3$
$$
\f_{\r_2}(m,k,n)=\Tr\left[
\Matt 0{e^{\pi i m}}0{e^{\pi i k}}0000{e^{\pi i(k+m)}}
\Matt 010001100
\right]=
$$
$$
=
\Tr\Matt 00{e^{\pi i m}}0{e^{\pi i k}}0{e^{\pi i(k+m)}}00
=e^{\pi i k}=\left\{
\begin{array}{rl}
1,& \mbox{ if } k \mbox{ is even }\\
-1,& \mbox{ if } k \mbox{ is odd }
\end{array}
\right. .
$$
$\f_{\r_2}$ is $3$-periodical in $n$, while the characteristic functions
of $B_i$ are $6$-periodical. For $j=0,\dots,5$ one has
$$
\begin{array}{llll}
\f_{\r_2}|_{B_1\cap L_0}\equiv 1, &
\f_{\r_2}|_{B_2\cap L_0}\equiv -1, &
\f_{\r_2}|_{B_3\cap L_1}\equiv 1, &
\f_{\r_2}|_{B_4\cap L_1}\equiv -1, \\
\f_{\r_2}|_{B_1\cap L_2}\equiv 1, &
\f_{\r_2}|_{B_2\cap L_2}\equiv -1, &
\f_{\r_2}|_{B_3\cap L_3}\equiv 1, &
\f_{\r_2}|_{B_4\cap L_3}\equiv -1, \\
\f_{\r_2}|_{B_1\cap L_4}\equiv 1, &
\f_{\r_2}|_{B_2\cap L_4}\equiv -1, &
\f_{\r_2}|_{B_3\cap L_5}\equiv 1, &
\f_{\r_2}|_{B_4\cap L_5}\equiv -1,
\end{array}
$$
so $\f_{\r_2}|_{B_1\cup B_3}\equiv 1$, $\f_{\r_2}|_{B_2\cup B_4}\equiv -1$.
The determinant of the values of the functions $\f_{\r_1}$, $\f_\pi$,
$\f_{\r_2}$, $\f_{\r_2\otimes\pi}$ on $B_1$, $B_2$, $B_3$, $B_4$ is
$$
\det \left(
\begin{array}{rrrr}
1 & 1 & 1 & 1 \\
1 & 1 & -1 & -1 \\
1 & -1 & 1 & -1 \\
1 & -1 & -1 & 1
\end{array}
\right)
=-8\ne 0.
$$
Hence, they are linearly independent.

Nevertheless, there are infinite-dimensional irreducible
$\wh\phi$-invariant representations. E.g. we have a representation
$\r$ of $G$ on $L^2(\T^2)$ with the respect to the Lebesgue measure,
with $\r(m,k,0)$ be the multiplier by characters in the appropriate
points and $\r(0,0,1)$ is $\a$, in the same manner as for $\r_2$.

This disproves the conjecture of Fel'shtyn and Hill \cite{FelHill},
who supposed
that the Reidemeister number equals to the number of fixed points of
$\wh \phi$ on $\wh G$.

This representation is not traceable, but
one can nevertheless try to calculate (\ref{eq:trsg}). Let us choose
an orthonormal base of $L^2(\T^2)$ formed by
$\e_{st}(x,y)=e^{2\pi i(sx+ty)}$, $x,y\in [0,1]$. The intertwining
operator $S$ is generated by $\mu$. Then
$$
\Tr(S\r(m,k,n))=\sum_{s,t} \int_0^1\int_0^1 (\r(m,k,0)
\r(0,0,n) \e_{s,t})(\mu(x,y))
\ov{\e_{s,t}}(x,y) \,dx\,dy
$$
$$
=\sum_{s,t} \int_0^1\int_0^1 e^{2\pi i \<(m,k),\mu(x,y)\>}
(\r(0,0,n) \e_{s,t})(\mu(x,y))
\e_{-s,-t}(x,y) \,dx\,dy
$$
$$
=\sum_{s,t} \int_0^1\int_0^1 e^{2\pi i \<(m,k),\mu(x,y)\>}
\e_{s,t}(\a^{n}\mu(x,y))
\e_{-s,-t}(x,y) \,dx\,dy
$$
$$
=\sum_{s,t} \int_0^1\int_0^1
e^{2\pi i \<(m,k)+(\a^{n}-\mu)(s,t),\mu(x,y)\>} \,dx\,dy=
\sum_{(m,k)=(\mu-\a^n)(s,t)} 1.
$$
For $n=0$ this equals $1$ for $m+k$ even and $0$ for $m+k$ odd.
Hence, $\f_\r|_{B_1\cap L_0}=1$, $\f_\r|_{B_2\cap L_0}=0$. For $n=1$
the equality takes the form $(m,k)=(-2s-2t,-t)$. Hence,
$\f_\r|_{B_3\cap L_1}=1$, $\f_\r|_{B_4\cap L_1}=0$.
If $n=2r$, the equality takes form
$$
(m,k)=(\mu-\a^{2r})(s,t),\quad
(m,k)=(\a^r \mu \a^r -\a^{2r})(s,t),
$$
$$
\a^{-r} (m,k)=(\mu-\a^0) \a^{r}(s,t).
$$
Since $\a$ is an automorphism of $\Z\oplus\Z$ and by the description
of $B_i$ via the action of $\a$, we obtain that
$\f_\r|_{B_1}=1$, $\f_\r|_{B_2}=0$. Similarly, for $n$ odd.
So, $\f_\r$ is well defined and
$$
\f_\r=\chi_{B_1}+\chi_{B_3}=\frac 12 (\f_{\r_1}+\f_{\r_2}).
$$

Of course, there are also $\wh\phi$-invariant traceable factor
representations of this group $G$, e.g. the regular representation.
Since its kernel is trivial, evidently all twisted-invariant
functionals can be pushed back from it.

This example shows that the conjecture of \cite{FelHill} for general
groups can survive only after eliminating badly separated points in
$\wh G$.

\section{Algebras of operator fields}\label{sec:algfiel}

Let us first remind the theory of operator fields following
\cite{Fell}. Let $T$ be a topological space and for each point $t\in T$
a $C^*$-algebra (or more general --- involutive Banach algebra)
$A_t$ is fixed.

\begin{dfn}
{\em A continuity structure for $T$ and the $\{A_t\}$\/} is a linear
space $F$
of operator fields on $T$, with values in the $\{A_t\}$,
(i.e. maps sending $t\in T$ to an element of $A_t$),
satisfying
\begin{enumerate}
\item if $x\in F$, the real-valued function $t\mapsto \|x(t)\|$ is continuous
on $T$;
\item for each $t\in T$, $\{x(t)\,|\,x\in F\}$ is dense in $A_t$;
\item $F$ is closed under pointwise multiplication and involution.
\end{enumerate}
\end{dfn}

\begin{dfn}
An operator field $x$ is {\em continuous\/} with respect to $F$ at $t_0$,
if for each $\e>0$, there is an element $y\in F$ and a neighborhood $U$
of $t_0$ such that $\|x(t)-y(t)\|<\e$ for all $t\in U$. The field $x$ is
{\em continuous on\/} $T$ if it is continuous at all points of $T$.
\end{dfn}

\begin{dfn}
A {\em full algebra of operator fields\/} is a family $A$ of operator
fields on $T$ satisfying:
\begin{enumerate}
\item $A$ is a *-algebra, i.e., it is closed under all the pointwise
algebraic operations;
\item for each $x\in A$, the function $t\mapsto \|x(t)\|$ is continuous
on $T$ and vanishes at infinity;
\item for each $t$, $\{x(t)\; | \; x\in A\}$ is dense in $A_t$;
\item $A$ is complete in the norm $\|x\|=\sup\limits_t\|x(t)\|$.
\end{enumerate}

\end{dfn}

A full algebra of operator fields is evidently a continuity
structure. If $F$ is any continuity structure, let us define
$C_0(F)$ to be the family of all operator fields $x$ which are
continuous on $T$ with respect to $F$, and for which $t\mapsto
\|x\|$ vanishes at infinity. One can prove that
$C_0(F)$ is a full algebra of operator fields --- indeed, a
maximal one.

\begin{lem}
For any full algebra $A$ of operator fields on $T$, the following three
conditions are equivalent:
\begin{enumerate}
\item $A$ is a maximal full algebra of operator fields;
\item $A=C_0(F)$ for some continuity structure $F$;
\item $A=C_0(A)$.
\end{enumerate}
\end{lem}

Such a maximal full algebra $A$ of operator fields may sometimes be called a
{\em continuous direct sum\/} of the $\{A_t\}$. It is clearly
{\em separating,\/} in the sense that , if $s,t\in T$, $s\ne t$, $\a\in A_s$,
$\b\in A_t$, there is an $x\in A$ such that $x(s)=\a$, $x(t)=\b$.
We will denote by $\wh a$ the section corresponding to an element
$a\in A$. We will study the unital case and $T$ will be compact without the
property of vanishing at infinity. The corresponding algebra will be denoted by
$\G (\A)\cong A$. Moreover, we suppose that $T$ is Hausdorff, hence,
normal.

\section{Functionals and measures}\label{sec:funstmeas}

\begin{dfn}\label{dfn:aofm}
A ({\em bounded additive\/}) {\em algebra of operator field measure\/}
(BA AOFM)
related to a maximal full algebra
of operator fields $A=\Gamma(\A)$ is a set function
$\mu:S\to \Gamma(\A)^*=A^*$, where $S\in \Sigma$, some algebra of sets,
\begin{itemize}
\item being additive:
$\mu(\sqcup S_i)(a)=\sum_i \mu(S_i)(a)$
\item $\mu(S)(a)=0$ if $\supp \wh a \cap S=\emptyset$.
\item bounded: $\sup$ over partitions $\{S_i\}$ of $T$ of
$\sum_i \|\mu(S_i)\|$ is finite and denoted by $\|\mu\|$
\end{itemize}
\end{dfn}

\begin{dfn}\label{dfn:regul}
It is *-{\em weak regular\/} (RBA AOFM)
if for each $E\in\Sigma$,
$a\in A$,
and $\e>0$ there is a
set $F\in\Sigma$ whose closure is contained in $E$ and a set $G\in\Sigma$
whose interior contains $E$ such that $\|\mu(C)a\|<\e$ for every
$C\in \Sigma$ with $C\ss G\setminus F$.
\end{dfn}

We will use as $\Sigma$ all sets and the algebra generated
by closed sets in $T$.

\begin{dfn}\label{dfn:lambmn}
Let AOFM $\l$ be defined on an algebra $\Sigma$ of sets in $T$
and $\l(\emptyset)=0$. A set $E\in\Sigma$
is called $\l$-{\em set\/}
if for any $M\in\Sigma$
$$
\l(M)=\l(M\cap E)+\l(M\cap(T\setminus E)).
$$
\end{dfn}

\begin{lem}\label{lem:mumn}
Let $\l$ be an {\rm AOFM} defined on an algebra $\Sigma$ of sets in
$T$ with $\l(\emptyset)=0$. The family of $\l$-sets is a subalgebra
of $\Sigma$ on which $\l$ is additive.
Furthemore, if $E$ is the union of a finite sequence $\{E_n\}$ of
disjoint $\l$-sets and $M\in \Sigma$, then
$\l(M\cap E)=\sum_n \l(M\cap E_n).$
\end{lem}

\begin{proof}
It is clear, that the void set, the whole space, and the
complement of any $\l$-set are $\l$-sets. Now let $X$ and $Y$ be
$\l$-sets, and $M\in\Sigma$. Then, since $X$ is a $\l$-set,
\begin{equation}\label{eq:i}
\l(M\cap Y)=\l(M\cap Y \cap X)+\l(M\cap Y\cap (T\setminus X)),
\end{equation}
and since $Y$ is a $\l$-set,
\begin{equation}\label{eq:ii}
\l(M)=\l(M\cap Y)+\l(M\cap (T\setminus Y)),
\end{equation}
$$
\l(M\cap (T\setminus (X\cap Y)))=
\l(M\cap (T\setminus (X\cap Y))\cap Y)+
\l(M\cap (T\setminus (X\cap Y))\cap (T\setminus Y)),
$$
hence,
\begin{equation}\label{eq:iii}
\l(M\cap (T\setminus (X\cap Y)))=
\l(M\cap (T\setminus X)\cap Y)+
\l(M\cap (T\setminus Y)).
\end{equation}
From (\ref{eq:i}) and (\ref{eq:ii}) it follows that
$$
\l(M)=\l(M\cap Y \cap X)+\l(M\cap Y\cap (T\setminus X))
+\l(M\cap (T\setminus Y)),
$$
and from (\ref{eq:iii}) that
$$
\l(M)=\l(M\cap Y \cap X)+\l(M\cap (T\setminus (X\cap Y))).
$$
Thus $X\cap Y$ is a $\l$-set. Since
$\cup X_n=T\setminus \cap (T\setminus X_n)$, we conclude that
the $\l$-sets form an algebra. Now if $E_1$, and $E_2$ are disjoint
$\l$-sets, it follows, by replacing $M$ by $M\cap(E_1\cup E_2)$
in Definition \ref{dfn:lambmn}, that
$$
\l(M\cap(E_1\cup E_2))=\l(M\cap E_1)+\l(M\cap E_2).
$$
The final conclusion of the lemma follows from this by induction.
\end{proof}

As it is well known, any functional $\t$ on a $C^*$-algebra $B$
can be represented as a linear combination of 4 positive functionals
in the following canonical way. First let us represent
$\t$ under the form $\t=\t_1+ i\t_2$, where self-adjoint functionals
$\t_1$ and $\t_2$ are defined by the formulas
\begin{equation}\label{eq:razlnasamosop}
\t_1(a)=\frac{\t(a)+\ov{\t(a^*)}}2,\qquad \t_2(a)=\frac{\t(a)-\ov{\t(a^*)}}{2i}.
\end{equation}
By the lemma about Jordan decomposition, any self-adjoint functional
$\a$ can be represented in a unique way as a difference of two positive
functionals $\a=\a_+-\a_-$ under requirement
\begin{equation}\label{eq:razlozhjord}
    \|\a\|=\|\a_+\|+\|\a_-\|
\end{equation}
(see \cite[\S 3.3]{Murphy},
\cite[Theorem 3.2.5]{Ped}). Let us decompose an AOFM in the related way.
Since the decomposition is unique, the second property of AOFM will held.
If we start from BA AOFM, then the additivity of summands will follow from
the uniqueness of the decomposition, and the boundedness (with double constant)
will follow from (\ref{eq:razlnasamosop}) and property
(\ref{eq:razlozhjord}).
The same argument shows that the summands are *-weak regular, if the initial
AOFM was *-weak regular.
So, the AOFM's in the decomposition are
{\em positive\/} AOFM, i.e., such that
$$
\mu(E)[a^*a]\ge 0
$$
for any $E\in \Sigma$.
Such a set function is non-decreasing.

\begin{lem}\label{lem:weakreg1}
The sets $F$ and $G$ in the {\rm Definition~\ref{dfn:regul}}
can be chosen in such a way that
$\|\mu(C)(fa)\|<\e$ for any continuous
function $f:T\to [0,1]$.
\end{lem}

\begin{proof}
Let us take the decompositions $\mu=\sum_{i=1}^4 x_i \mu_i$,
$a=\sum_{j=1}^4 y_j a_j$, where $\mu_i$ and $a_j$ are positive,
$x_i$, $y_j$ are complex numbers of norm $\le 1$. Let us choose
the sets $F$ and $G$, as in Definition~\ref{dfn:regul}, for
$\e/16$ and for all pairs $\mu_i$, $a_j$ simultaneously. Then
$$
0\le
\mu_i(C)(f\cdot a_j)=\mu_i(C)((a_j)^{1/2} f(a_j)^{1/2})\le
\mu_i(C)(a_j)\le \frac \e{16},
$$
and
$$
\|\mu_i(C)(f\cdot a)\|\le \sum_{i,j=1}^4 |x_i y_j|\cdot
|\mu_i(C)(a_j)|\le 16\cdot\frac \e{16}=\e.
$$
\end{proof}

\begin{teo}\label{teo:razryv}
Let a unital separable $C^*$-algebra $A$ be isomorphic to
a full algebra of operator fields $\G(\A)$ over a Hausdorff space $X$.
Then functionals on $A\cong \G(\A)$ can be identified with {\rm RBA AOFM\/}
of $\A$.
\end{teo}

\begin{proof}
Let us remark that these suppositions imply the following: $T$ is a
separable Hausdorff compact and the unit ball of the dual space of $A$
is a metrizable compact in $*$-weak topology.

Obviously RBA AOFM form a linear normed space with respect to
$\|.\|$.

First, we want to prove that the natural linear map $\mu\mapsto \mu(T)$ is
an isometry of RBA AOFM into $A^*$. Since $\|\mu(T)\|\le \|\mu\|$, it is
of norm $\le 1$. Let now take an arbitrary small $\e>0$. Let $E_1,\dots,
E_n$ be a partition of $T$ such that
$$
\sum_{i=1}^n \|\mu(E_i)\|\ge \|\mu\|-\e.
$$
Let $a_i\in A$ be such elements of norm 1, that
$\mu(E_i)(a_i)\ge \|\mu(E_i)\|-\e/n$.

By *-weak regularity of $\mu$ and normality of $T$ one can take closed
sets $C_i$, disjoint open sets $G_i$, and continuous functions
$f_i:T\to [0,1]$ such that $C_i \ss E_i$,
$\|\mu(E_i\setminus C_i)(a_j)\|\le \e/n^2$,
$C_i\ss G_i$, $\|\mu(G_i\setminus C_i)(a_j)\|\le \e/n^2$,
(and estimations hold for multiplication by positive functions
as well, as in Lemma~\ref{lem:weakreg1})
$f_i(s)=0$ if $s\not\in G_i$, $f_i(s)=1$ if $s\in C_i$, $i,j=1,\dots,n$.

Consider the element
$a:=\sum_i f_i a_i \in \G(\A)=A$. Then $\|a\|\le 1$ and
$$
|\:\mu(S)(a)-\|\mu\|\:|\le \sum_{i=1}^n\left| \mu(E_i)(a)-\mu(E_i) \right|+\e
$$
$$
\le \sum_{i=1}^n\left| \mu(E_i\setminus C_i)(a) + \mu(C_i)(a)
-\mu(E_i)(a_i) \right|+2\e
$$
$$
= \sum_{i=1}^n\left|  \sum_{j=1}^n\mu(E_i\setminus C_i)(f_j a_j)
+ \mu(C_i)(a_i)-\mu(E_i)(a_i) \right|+2\e
$$
$$
\le \sum_{i,j=1}^n | \mu(E_i\setminus C_i)(f_j a_j)|
+ \sum_{i=1}^n \left| \mu(E_i\setminus C_i)(a_i) \right|+2\e
\le n^2 \:\frac \e{n^2} + n \:\frac \e{n^2} + 2\e
\le 4\e.
$$
Since $\e$ is arbitrary small, $\|\mu\|=\|\mu(S)\|$.

It remains to represent the general functional $\f$ by RBA AOFM.
This functional on $\G(\A)$ can be extended by Hahn-Banach
theorem to a continuous functional $\psi$ on $B(\A)=\prod_{t\in T}A_t$
(the $C^*$-algebra
of possibly discontinuous cross-sections of $\A$ with $\sup$-norm).
This functional can be decomposed
$\psi=\sum_{i=1}^4 \a_i\psi_i$, where $\psi_i$ are
positive functionals, $\a_i\in\C$, $|\a_i|=1$, $\|\psi_i\|\le \|\psi\|$.
Let us define
$$
\l(E)(a):=\psi(\chi_E a),\qquad \l_i(a):=\psi_i(\chi_E a),\quad
i=1,\dots,4.
$$
where $a\in \G(\A)$, and $\chi_E$ is the characteristic function
of $E$. Evidently, $\l(T)(a)=\psi(a)$ and $\l$ is a BA AOFM.
Indeed, the first two properties of Definition \ref{dfn:aofm} are
evident. The third one can be verified for each $\l_i$,
$i=1,\dots,4$:
$$
\sum_{j=1}^N |\l_i(E_j)|=\sum_{j=1}^N \l_i(E_j)(\1)=
\sum_{j=1}^N \psi_i(\chi_{E_j}\1)=\psi_i(\1)\le \|\psi_i\|,
$$
hence,
$$
\sum_{j=1}^N |\l(E_j)| \le \sum_{i=1}^4 \sum_{j=1}^N
|\l_i(E_j)|\le \sum_{i=1}^4 \|\psi_i\| \le 4\cdot \|\psi\|.
$$

Now we want to find an RBA AOFM $\mu$ such that
$\mu(T)(a)=\l(T)(a)$. Without loss of generality it is sufficient
to do this for a positive $\l=\l_i$.

Let $F$ represent the general closed subset, $G$ the general
open subset, $E$ the general subset of $S$. Define $\mu_1$ and
$\mu_2$ by putting
$$
\mu_1(F)(a^*a)=\inf_{G\supset F} \l(G) (a^*a),\qquad
\mu_2(E) (a^*a)=\sup_{F\ss E} \mu_1(F) (a^*a),
$$
and then by taking linear extension.
More precisely, due to separability one can take a cofinal
sequence $\{G_i\}$. The unit ball in the dual space is weakly compact
and one can take a weakly convergent sequence $\l(G_{i_k})$.
Its limit $\psi$ is a positive functional on $A$ enjoying $\inf$-property
on positive elements. In particular, it is independent of the choice
of $\{G_i\}$ and $\{G_{i_k}\}$. In a similar way for $\sup$.

These set functions are non-negative and non-decreasing.
Let $G_1$ be open and $F_1$ be closed. Then if
$G\supset F_1\setminus G_1$, it follows that $G_1\cap G\supset F_1$
and $\l(G_1)\le \l(G_1)+\l(G)$ so that
$\mu_1(F_1)\le \l(G_1)+\l(G)$. Since $G$ is an arbitrary open
set containing $F_1\setminus G_1$ we have
$$
\mu_1(F_1)\le \l(G_1) +\mu_1(F_1\setminus G_1).
$$
If $F$ is a closed set it follows from this inequality, by allowing
$G_1$ to range over all open sets containing $F\cap F_1$, that
$$
\mu_1(F_1)\le \mu_1 (F\cap F_1) + \mu_2 (F_1\setminus F).
$$
If $E$ is an arbitrary subset of $T$ and $F_1$ ranges over the closed
subsets of $E$, then it follows from the preceding inequality that
\begin{equation}\label{eq:idanf}
\mu_2(E) \le \mu_2(E\cap F) + \mu_2(E\setminus F).
\end{equation}
It will next be shown that for an arbitrary set $E$ in $T$
and arbitrary closed set $F$ in $T$ we have
\begin{equation}\label{eq:iidanf}
\mu_2(E) \ge \mu_2(E\cap F) + \mu_2(E\setminus F).
\end{equation}
To see this let $F_1$ and $F_2$ be disjoint closed sets. Since $T$ is
normal there are disjoint neighborhoods $G_1$ and $G_2$ of $F_1$ and $F_2$
respectively. If $G$ is an arbitrary neighborhood of $F_1\cup F_2$
then $\l(G)\ge \l(G\cap G_1) + \l(G\cap G_2)$ so that
$$
\mu_1(F_1\cap F_2)\ge \mu_1 (F_1) + \mu_2(F_2).
$$
We now let $E$ and $F$ be arbitrary sets in $T$ with $F$ closed and
let $F_1$  range over closed subsets of $E\cap F$ while $F_2$ ranges over
the closed subsets of $E\setminus F$. The preceding inequality then
proves (\ref{eq:iidanf}). From (\ref{eq:idanf}) and (\ref{eq:iidanf})
we have
\begin{equation}\label{eq:iiidanf}
\mu_2(E)=\mu_2(E\cap F) + \mu_2(E\cap (T\setminus F))
\end{equation}
for any $E$ in $T$ and $F$ closed.
The function $\mu_2$ is defined on the algebra of all subsets
of $T$ and it follows from (\ref{eq:iiidanf}) that every closed
set $F$ is a $\mu_2$-set.
If $\mu$ is the restriction of $\mu_2$ on the algebra determined
by the closed sets, it follows from Lemma \ref{lem:mumn}
that $\mu$ is additive
on this algebra. It is clear from the definition of $\mu_1$ and $\mu_2$
that $\mu_1(F)=\mu_2(F)=\mu(F)$ if $F$ is closed and thus
$\mu(E)=\sup_{F\ss E}\mu(F)$. This shows that $\mu$ is *-weak regular
and since $\|\mu(T)\|<\infty$, we have $\mu$ is RBA AOFM.

Finally, by the definition, $\mu(S)(a)=\l(S)(a)=\psi(a)=\f(a)$
for $a\in \G(\A)$.
\end{proof}

\section{Twisted-invariant AOFM}\label{sec:invfunsme}

The most part of the argument is valid for various representations
of algebras by operator fields, but now we will pass to the case
of group $C^*$-algebra of a discrete group and concentrate
ourselves on the following important case due to Dauns and
Hofmann \cite[Corollary 8.13]{DaunsHofmann}
(see also \cite[Corollary 8.14]{DaunsHofmann}):

\begin{teo}
Consider a $C^*$-algebra $A$ $($with or without an identity$)$
and the set $B$ of all its primitive ideals in the hull-kernel
topology $q$.

(1) The complete regularization $\f:(B,q)\to (M,t)$ can be taken
to consist of closed ideals $m$ of $A$ satisfying $m=\cap
\f^{-1}(m)$. The topology $t$ of $M$ contains the weak-star
topology $t^*$ induced on $M$ by $A$.

Denote by $\K$ the family of all $t$-compact subsets of $M$ of the
form $\{m\in M\: |\: \|a+m\|\ge \e\}$, $a\in A$, $0<\e$. Let
$\pi'':E'\to (M,t)$ be the uniform field of $C^*$-algebras
obtained by first forming the canonical field from $A$, $M$ and
then enlarging the weak-star topology $t^*$ of $M$ up to $t$. For
each $a\in A$, $\hat a$ is the map $\hat a:M\to E'$,
$\hat a(m)=a+m\in A/m$. The $C^*$-algebra of all sections
vanishing at infinity with respect to the class $\K$ is denoted by
$\G_0(\pi'')$.

(2) Then the map $A\to \G_0(\pi'')$, $a\mapsto \hat a$ is an
isometric star-isomorphism. In particular, if $M$ is $t$-compact,
then $M\in \K$ and $A\cong \G(\pi'')$.
\end{teo}

\begin{dfn}\label{dfn:glimmideal}
Ideals $m$ from the previous theorem, i.e., points of $M$, are
called \emph{Glimm ideals}.
\end{dfn}

We will use the notation $T$ instead of $M$ for Glimm spectrum
to agree with the previous section.

Now we consider a countable discrete group $G$ and its
automorphism $\phi$. Also, $A=C^*(G)$. One has the twisted action
of $G$ on $A$:
$$
g[a]=L_g a L_{\phi(g^{-1})},
$$
where $L_g$ is the left shift, and respectively on functionals.

\begin{dfn}\label{dfn:generReid}
The dimension $R_*(\phi)$ of the space of twisted invariant functionals
on $C^*(G)$ is called \emph{generalized Reidemeister number of $\phi$}.
\end{dfn}

\begin{dfn}
A (Glimm) ideal $I$ is a {\em generalized fixed point of $\wh \phi$\/}
if the linear span of
elements $b-g[b]$ is not dense in $A_I=A/I$, where
$g[.]$ is the twisted action, i.e. its closure $K_I$ does not coincide
with $A_I$.
\end{dfn}

If we have only a finite number of such fixed
points, then the twisted invariant RBA AOFM's are
concentrated in these points. Indeed, let us describe the action of $G$
on RBA AOFM's in
more detail.
The action of $G$ on measures is defined by the identification of
the measures and functionals on $A$.

\begin{lem}\label{lem:ograninv}
If $\mu$ corresponds to a twisted invariant functional, then
for any Borel $E\ss T$ the functional $\mu(E)$ is twisted
invariant.
\end{lem}

\begin{proof}
This is an immediate consequence of *-weak regularity.
Let
$a\in A$, $g\in G$  and
$\e>0$ be an arbitrary small number, and $U$ and $F$ as in
Lemma~\ref{lem:weakreg1}, for $a$ and $g[a]$ simultaneously.
Let us choose a continuous
function $f:T\to [0,1]$, with $f|_F=1$ and $f|_{T\setminus U}=0$.
Then
$$
|\mu(E)(a-g[a])|=
|\mu(E\setminus F)(a)
+ \mu(F)(a)
-\mu(E\setminus F)(g[a])
- \mu(F)(g[a])|
$$
$$
\le |\mu(F)(a)-\mu(F)(g[a])|+2\e =|\mu(F)(fa)-\mu(F)(g[fa])|+2\e
$$
$$
=|\mu(U)(fa)-\mu(U\setminus F)(fa)-\mu(U)(g[fa])
+\mu(U\setminus F)(g[fa])|+2\e
$$
$$
\le |\mu(U)(fa)-\mu(U)(g[fa])|+4\e =|\mu(T)(fa)-\mu(T)(g[fa])|+4\e =4\e.
$$
\end{proof}

\begin{lem}
For any twisted-invariant functional $\f$ on $C^*(G)$ the corresponding
measure $\mu$ is concentrated in the set $GFP$ of generalized fixed points.
\end{lem}

\begin{proof}
Let $\|\mu\|=1$. Suppose opposite: there exists an element $a\in A$,
$\|a\|=1$, vanishing on generalized fixed points, such that $\f(a)\ne 0$.
Let $\e:=|\f(a)|>0$. In each point $t\not\in GFP$ we can find elements
$b^i_t\in A$, $g^i_t\in G$, $i=1,\dots k(t)$,
such that
$$
\|a(t)-\sum_{i=1}^{k(t)}(g^i_t[b^i_t](t)-b^i_t(t))\|<\e/4.
$$
Then there exists a
neighborhood $U_t$ such that for $s\in U_t$ one has
$$
\|a(s)-\sum_{i=1}^{k(t)}(g^i_t[b^i_t](s)-b^i_t(s))\|<\e/2.
$$
  Let us choose a finite subcovering
$\{U_{t_j}\}$, $j=1,\dots,n$,
of $\{U_t\}$
and a Borel partition $E_1,\dots,E_n$ subordinated to this subcovering.
Then
$$
\f(a)=\sum_{j=1}^n \mu(E_j)(a)=
\sum_{j=1}^n \mu(E_j)\left(a-\sum_{i=1}^{k(t_j)}(g^i_{t_j}[b^i_{t_j}]
-b^i_{t_j})\right)
+\sum_{j=1}^n \sum_{i=1}^{k(t_j)}\mu(E_j)(g^i_{t_j}[b^i_{t_j}]
-b^i_{t_j}).
$$
By Lemma \ref{lem:ograninv} each summand in the second term is zero.
The absolute value of the first term is less then
$\sum_j\|\mu(E_j)\| \e/2 \le \|\mu\|\cdot \e/2=\e/2$.
A contradiction with $|\f(a)|=\e$.
\end{proof}

Since a functional $\f$ on $A_I$ is twisted-invariant if and only if
$\Ker\f\supset K_I$, the dimension of the space of these functionals
equals the dimension of the space of functionals on $A_I/K_I$ and is
finite  if and only if the space $A_I/K_I$ is
finite dimensional. In this case the dimension of the space of
twisted-invariant functionals on $A_I$ equals $\dim(A_I/K_I)$.

\begin{dfn}
\emph{Generalized number $S_*(\phi)$ of fixed points} of $\wh\phi$
on Glimm spectrum is
$$
S_*(\phi):=\sum_{I\in GFP} \dim (A_I/K_I).
$$
\end{dfn}

Since functionals associated with measures,
which are concentrated in different points,
are linearly independent (the space is Hausdorff),
the argument above gives the following statement.
\begin{teo}[Weak generalized Burnside]
$$
R_*(\phi)=S_*(\phi)
$$
if one of these numbers is finite.
\end{teo}

In~\cite{FelTro} (see also~\cite{FelshB,FelTroVer}
and the first part of the present survey paper) it was proved
for groups of type I that the Reidemeister number $R(\phi)$, i.e.,
the number of twisted (or $\phi$-)conjugacy classes, coincides with
the number of fixed points of the corresponding action $\wh\phi$ on
the dual space $\wh G$. This was a generalization of Burnside theorem.
The proof used identification of $R(\phi)$ and the dimension of the
space of ($L^\infty$-) twisted invariant functions on $G$, i.e. twisted
invariant functionals on $L^1(G)$. Since only part of $L^\infty$-functions
defines functionals on $C^*(G)$ (namely, Fourier-Stieltjes functions),
so, \emph{a priori} one has
$R_*(\phi)\le R(\phi)$. Nevertheless, the functions with some
symmetry conditions very often are in  Fourier-Stieltjes algebra,
so one can conjecture that $R(\phi)=R_*(\phi)$ if $R(\phi)<\infty$.
This is the case for all known examples.

\section{Finite-dimensional representations and almost polycyclic groups} \label{sec:polyc}


Let us start from several useful facts.

\begin{teo}[{\cite[Theorem 1.41]{Robinson72-1}}]\label{teo:fingeneinind}
If $G$ is a finitely generated group and $H$ is a subgroup with
finite index in $G$, then $H$ is finitely generated.
\end{teo}

\begin{lem}\label{lem:normalcharfiniteind}
Let $G$ be finitely generated, and $H'\ss G$
its subgroup of finite index. Then there is a characteristic
subgroup $H\ss G$ of finite index, $H\ss H'$.
\end{lem}

\begin{proof}
Since $G$ is finitely generated, there is only finitely many
subgroups of the same index as $H'$
(see \cite{Hall}, \cite[\S\ 38]{Kurosh}).
Let $H$ be their intersection. Then $H$ is characteristic, in
particular normal, and of finite index.
\end{proof}

\begin{teo}[see \cite{Jiang}]\label{teo:reidandcokerabel}
Let $A$ be a finitely generated Abelian group, $\psi:A\to A$ its
automorphism. Then $R(\psi)=\#\Coker(\psi-\Id)$, i.e. to the index
of subgroup generated by elements of the form $x^{-1}\psi(x)$.
\end{teo}

\begin{proof}
By Lemma \ref{lem:charabelclassred}, $R(\psi)$ is equal to the index of the subgroup
$H=\{e\}_\psi$. This group consists by definition
of elements of the form $x^{-1}\psi(x)$.
\end{proof}


Let us remind the following definitions of a class of groups.

\begin{dfn}\label{dfn:FCgroup}
A group with finite conjugacy classes is called FC-\emph{group}.
\end{dfn}

In a FC-group the elements of finite order form a characteristic
subgroup with locally infinite abelian factor group;
a finitely generated FC-group contains in its center a free abelian
group of finite index in the whole group \cite{BNeumann}.

\begin{lem}\label{lem:fgFChasfinitefix}
An automorphism $\phi$ of a finitely generated {\rm FC}-group $G$
with $R(\phi)<\infty$ has a finite number of fixed points.

The same is true for $\t_x\circ \phi$. Hence, the number of
$g\in G$ such that for some $x\in G$
$$
g x \phi(g^{-1})=x,
$$
is still finite.
\end{lem}

\begin{proof}
Let $A$ be the center of $G$. As it was indicated, $A$ has a finite
index in $G$ and hence, by Theorem \ref{teo:fingeneinind}, is f.g.
Since $A$ is characteristic, one has an extension $A\to G\to G/A$
respecting $\phi$. One has $R(\phi')\le R(\phi)\cdot |G/A|$
by Lemma \ref{lem:ozenkacherezperiod}
and $\#\Fix(\phi')\le R(\phi)\cdot |G/A|$ by Theorem
\ref{teo:fixandreidforabel}. Then $\#\Fix(\phi)\le R(\phi)\cdot |G/A|^2$
by Lemma \ref{lem:fixedofextens}.
\end{proof}

\begin{dfn}\label{dfn:Rperiodi}
We say that a group $G$ has the property   {\sc RP} if for
any automorphism $\phi$ with $R(\phi)<\infty$ the
characteristic functions $f$ of {\sc Reidemeister} classes (hence
all $\phi$-central functions) are   {\sc periodic} in the
following sense.

There exists a finite group $W$,
its automorphism $\phi_W$, and epimorphism $F:G\to W$ such that
\begin{enumerate}
    \item The diagram
    $$
    \xymatrix{
G\ar[r]^\phi\ar[d]_F& G\ar[d]^F\\
W\ar[r]^{\phi_W}& W
    }
    $$
    commutes.
    \item $f=F^*f_W$, where $f_W$ is a characteristic function
    of a subset of $W$.
\end{enumerate}
If this property holds for a concrete automorphism $\phi$, we
will denote this by RP($\phi$).
\end{dfn}

\begin{rk}\label{rk:bijecofclass}
By (2) there is only one class $\{g\}_{\phi}$ which maps onto $\{F(g)\}_{\phi_W}$.
Hence, $F$ induces a bijection of Reidemeister classes.
\end{rk}

\begin{lem}\label{lem:periodihomom}
Suppose, $G$ is f.g.
and $R(\phi)<\infty$. Then
characteristic functions of $\phi$-conjugacy classes
are periodic $($i.e. $G$ has {\rm RP}$(\phi)$ $)$
if and only if their left
shifts generate a finite dimensional space.
\end{lem}

\begin{proof}
From the supposition of finite dimension it follows that the
stabilizer of each $\phi$-conjugacy class has finite index.
Hence, the common stabilizer of all $\phi$-conjugacy classes
under left shifts is an intersection of finitely many subgroups, each
of finite index. Hence, its index is finite.
By Lemma \ref{lem:normalcharfiniteind}
there is some smaller subgroup $G_S$ of finite index
which is normal and $\phi$-invariant.
Then one can take $W=G/G_S$. Indeed, it is sufficient to verify
that the projection $F$ is one to one on classes. In other words,
that each coset of $G_s$ enters only one $\phi$-conjugacy class,
or any two elements of coset are $\phi$-conjugated. Consider
$g$ and $hg$, $g\in G$, $h\in G_s$. Since $h$ by definition preserves
classes, $hg=xg\phi(x^{-1})$ for some $x\in G$, as desired.

Conversely, if  $G$ has {\rm RP}$(\phi)$, the class $\{g\}_\phi$ is
a full pre-image $F^{-1}(S)$ of some class $S\ss W$. Then its left shift
can be described as
\begin{eqnarray*}
g'\{g\}_\phi &=& g'F^{-1}(S)=\{g'g_1 | g_1\in F^{-1}(S)\}=\{g | (g')^{-1}g\in F^{-1}(S)\}\\
&=&
\{g | F((g')^{-1}g)\in S\}=\{g | F(g)\in F(g')(S)\}=F^{-1}(F(g')(S)).
\end{eqnarray*}
Since $W$ is finite, the number of these sets is finite.
\end{proof}

\begin{rk}\label{rk:odnodliavseh}
1) In this situation in accordance with  Lemma \ref{lem:normalcharfiniteind}
the subgroup $G_S$ is characteristic, i.e. invariant under any automorphism.

2) Also, the group $G/G_S$ will serve as $W$
(i.e. give rise a bijection on sets of Reidemeister classes)
not only for $\phi$ but for
$\t_g\circ \phi$ for any $g\in G$ because they have the same collection
of left shifts of Reidemeister classes by Lemma \ref{lem:redklassed}.
\end{rk}

\begin{teo}\label{teo:RPFCexten}
Suppose, the extension {\rm (\ref{eq:extens})} satisfies
the following conditions:
\begin{enumerate}
    \item $H$ has {\rm RP};
    \item $G/H$ is {\rm FC} f.g.
\end{enumerate}
Then $G$ has {\rm RP($\phi$)}.
\end{teo}

\begin{proof}
We have $R(\ov\phi)<\infty$, hence $\#\Fix(\ov\phi)<\infty$
by Lemma \ref{lem:fgFChasfinitefix} as well as $\#\Fix(\t_z\ov\phi)<\infty$ for any $z\in G/H$.
Then by Lemma \ref{lem:ozenkacherezperiod} $R(\t_g\phi')<\infty$ for any $g\in G$.
Let $g_1,...,g_s$, $s=R(\ov\phi)$, be elements of $G$ which are mapped by $p$
to different Reidemeister classes of $\ov\phi$.
Now we
can apply the supposition that $H$ has RP and find a characteristic subgroup
$H_W:=\Ker F\ss H$ of finite index such that $F:H\to W$ gives rise a bijection
for Reidemeister classes of each $\t_{g_i}\circ \phi'$, $i=1,\dots,s$, moreover,
it is contained in the stabilizer of each twisted conjugacy class of each of these
automorphisms.
We choose it in a way to satisfy Remark \ref{rk:odnodliavseh}.
In particular, it is normal in $G$.
Hence, we can take a quotient by $H_W$ of the extension
(\ref{eq:extens}):
$$
\xymatrix{
H\ar@{^{(}->}[r]\ar[d]^F&  G\ar[d]^{F_1} \ar[rd]^p &\\
H/H_W\ar@{^{(}->}[r]\ar@{=}[d]&  G/H_W\ar@{=}[d] \ar[r]^p & G/H\ar@{=}[d]\\
W\ar@{^{(}->}[r]& G_1 \ar[r]^p& \G.
}
$$
The quotient map $F_1: G\to G/H_W$ takes $\{g\}_\phi$ to
$\{g\}_\phi$ and it is a unique class with this property
(we conserve the notations $e,g,\phi$ for the quotient
objects). Indeed, suppose two classes are mapped onto one.
Hence, $g_i h$ and $g_i  h \wh h_W$ belong to the same (different) classes as $g$ and $gh_W$.
Moreover, they can not be $\phi$-conjugate by elements of $H$. Hence (cf. (\ref{eq:klasstaugsh})),
the elements $h$ and $h \wh h_W$ are not $(\t_{g_i}\circ \phi')$-conjugate in $H$. But this contradicts
the choice of $H_W\ni \wh h_W$.

By Lemma \ref{lem:periodihomom} (applied to $G_1$ and concrete
automorphism $\phi$) for the purpose to find a map $F_2:G_1\to W_1$
with properties (1) and (2) of the Definition \ref{dfn:Rperiodi}
it is sufficient to verify that shifts of the characteristic
function of $\{h\}_\phi\ss G_1$ form a finite dimensional space, i.e.
the shifts of  $\{h\}_\phi\ss G_1$
form a finite collection of subsets of $G_1$.
After that one can take the composition
$$
G\ola{F_1} G_1 \ola{F_2}  W_1
$$
to complete the proof of theorem.

Let us observe, that we can apply Lemma \ref{lem:periodihomom} because
the group $G_1$ is finitely generated: we can take
as generators all elements of $W$ and some pre-images $s(z_i)\in G_1$
under $p$ of a finite system of generators $z_i$ for $\G$.
Indeed, for any $x\in G_1$ one can find some product of $z_i$ to be
equal to $p(x)$. Then the same product of $s(z_i)$ differs from $x$ by
an element of $W$.

Let us prove that the mentioned space of shifts is finite-dimensional.
By Lemma \ref{lem:redklassed} these shifts of
$\{h\}_\phi\ss G_1$ form a subcollection of
$$
\{x\}_{\t_y\circ \phi},\quad x,y\in G_1.
$$
Hence, by Corollary \ref{cor:rphiequaltaurphi} it is sufficient to
verify that the number of different automorphisms $\t_y:G_1\to G_1$
is finite.

Let $x_1,\dots,x_n$ be some generators of $G_1$. Then the number of different
$\t_y$ does not exceed
$$
\prod_{j=1}^n \# \{\t_y(x_j)\:|\:y\in G_1\}\le
\prod_{j=1}^n |W|\cdot \# \{\t_{z}(p(x_j))\: |\: z\in \G\},
$$
where the last numbers are finite by the definition of FC for $\G$.
\end{proof}


Now we describe using Theorem \ref{teo:RPFCexten}
some classes of RP groups. Of course these classes are only
a small part of possible corollaries of this theorem.

Let $G'=[G,G]$ be the \emph{commutator subgroup} or
\emph{derived group} of $G$, i.e. the subgroup generated by
commutators. $G'$ is invariant under any homomorphism, in
particular it is normal. It is the smallest normal subgroup of $G$
with an abelian factor group. Denoting $G^{(0)}:=G$, $G^{(1)}:=G'$,
$G^{(n)}:=(G^{(n-1)})'$, $n\ge 2$, one obtains \emph{derived series}
of $G$:
\begin{equation}\label{eq:derivedseries}
    G=G^{(0)}\supset G'\supset G^{(2)}\supset \dots\supset G^{(n)}\supset
\dots
\end{equation}
If $G^{(n)}=e$ for some $n$, i.e. the series (\ref{eq:derivedseries})
    stabilizes by trivial group,
    the group $G$ is \emph{solvable};

\begin{dfn}\label{dfn:polycgroup}
A solvable group with derived series with cyclic factors is called
\emph{polycyclic group}.
\end{dfn}

\begin{teo}\label{teo:polyciclRP}
Any polycyclic group is {\rm RP}.
\end{teo}

\begin{proof}
By Lemma \ref{lem:charabelclassred} any commutative group is RP.
Any extension with $H$ being the commutator subgroup $G'$
of $G$ respects any
automorphism $\phi$ of $G$, because $G'$ is evidently characteristic.
The factor group is abelian, in particular FC.

Since any polycyclic group is a result of finitely many such extensions
with finitely generated (cyclic) factor groups,
starting from Abelian group,
applying inductively Theorem \ref{teo:RPFCexten}
we obtain the result.
\end{proof}

\begin{teo}\label{teo:nilpotRP}
Any finitely generated nilpotent group is {\rm RP}.
\end{teo}

\begin{proof}
These groups are supersolvable, hence, polycyclic \cite[5.4.6, 5.4.12]{Robinson}.
\end{proof}

\begin{teo}\label{teo:polynomRP}
Any finitely generated group of polynomial growth is {\rm RP}.
\end{teo}

\begin{proof}
By \cite{GromovPolyIHES} a  finitely generated group
of polynomial growth is just
as finite extension of a f.g. nilpotent group $H$.
The subgroup $H$ can be supposed to be characteristic, i.e.
$\phi(H)=H$ for any automorphism $\phi:G\to G$.
Indeed, let $H'\ss G$ be a nilpotent subgroup of index $j$.
Let $H$ be the subgroup from Lemma \ref{lem:normalcharfiniteind}.
By Theorem \ref{teo:fingeneinind} it is finitely generated.
Also, it is nilpotent as
a subgroup of nilpotent group (see \cite[\S\ 26]{Kurosh}).

Since a finite group is a particular case of FC group
and f.g. nilpotent group has RP by Theorem \ref{teo:nilpotRP},
we can apply Theorem \ref{teo:RPFCexten} to complete the proof.
\end{proof}

The proof of the following twisted Burnside
theorem can be extracted from
Theorem \ref{teo:twburnsforRP}, but we formulate it separately
due to importance of the classes under consideration.

\begin{teo}\label{teo:TwistedBurnsPolyn}
The Reidemeister number of any automorphism $\phi$ of a f.g. group
of polynomial growth or polycyclic group is equal to the
number of finite-dimensional
fixed points of $\wh\phi$ on the unitary dual of this group if $R(\phi)$
is finite.
\end{teo}

\begin{teo}\label{teo:RPalmostpolyc}
Any almost polycyclic group, i.e. an extension of
a polycyclic group with a finite factor group, is {\rm RP}.
\end{teo}

\begin{proof}
The proof repeats almost
literally the proof of Theorem \ref{teo:polynomRP}.
One has to use the fact that a subgroup of a polycyclic group
is polycyclic \cite[p.~147]{Robinson}.
\end{proof}


\begin{dfn}
Denote by $\wh G_f$ the subset of the unitary dual $\wh G$ related to
finite-di\-men\-si\-on\-al representations.
\end{dfn}

\begin{teo}[Twisted Burnside Theorem]\label{teo:twburnsforRP}
Let $G$ be an {\rm RP} group and $\phi$ its automorphism. Denote by $S_f(\phi)$
the number of fixed points of $\wh\phi_f$ on $\wh G_f$. Then
$$
R(\phi)=S_f(\phi),
$$
if the Reidemeister number $R(\phi)$ is finite.
\end{teo}

\begin{proof}
Let us start from the following observation.
Let $\Sigma$ be the universal compact group associated with $G$ and
$\a:G\to \Sigma$ the canonical morphism
(see, e.g. \cite[Sect. 16.1]{DixmierEng}).
Then $\wh G_f=\wh \Sigma$  \cite[16.1.3]{DixmierEng}.
The coefficients of (finite-dimensional) non-equivalent irreducible
representations of $\Sigma$ are linear independent by Peter-Weyl theorem
as functions on $\Sigma$. Hence the coefficients of
finite-dimensional non-equivalent irreducible
representations of $G$  as functions on $G$ are linearly independent as well.

It is sufficient to verify the following three statements:

1) If $R(\phi)<\infty$, than each $\phi$-class function is a finite
linear combination of twisted-invariant functionals being coefficients
of points of $\Fix\wh\phi_f$.

2) If $\rho\in\Fix\wh\phi_f$, there exists one and only one
(up to scaling)
twisted invariant functional on $\rho(C^*(G))$ (this is a finite
full matrix algebra).

3) For different $\rho$ the corresponding $\phi$-class functions are
linearly independent. This follows from the remark at the beginning of
the proof.

Let us remark that the property RP implies in particular that
$\phi$-central functions (for $\phi$ with $R(\phi)<\infty$) are
functionals on $C^*(G)$, not only $L^1(G)$, i.e. are in the
Fourier-Stieltijes algebra $B(G)$.

The statement 1) follows from the RP property. Indeed,
this $\phi$-class function $f$ is a linear combination of functionals
coming from some finite collection $\{\r_i\}$ of elements
of $\wh G_f$ (these
representations $\r_1,\dots,\r_s$ are in fact representations of
the form $\pi_i\circ F$, where $\pi_i$ are irreducible representations
of the finite group $W$ and $F:G\to W$, as in the definition of RP).
So,
$$
f=\sum_{i=1}^s f_i\circ \r_i,\quad \r_i:G\to \End (V_i),
\quad f_i:\End (V_i)\to \C, \quad \r_i\ne \r_j,\: (i\ne j).
$$
For any $g, \til g \in G$ one has
$$
\sum_{i=1}^s f_i (\r_i(\til g ))=f(\til g )=
f(g\til g \phi(g^{-1}))=\sum_{i=1}^s f_i (\r_i(g\til g \phi(g^{-1}))).
$$
By the observation at the beginning of the proof concerning
linear independence,
$$
f_i (\r_i(\til g ))=f_i (\r_i(g\til g \phi(g^{-1}))),\qquad i=1,\dots,s,
$$
i.e. $f_i$ are twisted-invariant.  For any $\r\in \wh G_f$,
$\r:G\to \End (V)$, any
functional $\w: \End (V)\to \C$
has the form $a\mapsto \Tr(ba)$ for some fixed $b\in \End (V)$.
Twisted invariance implies twisted invariance of $b$ (evident details can
be found in \cite[Sect. 3]{FelTro}). Hence, $b$ is intertwining between
$\r$ and $\r\circ\phi$ and $\r\in \Fix(\wh\phi_f)$. The uniqueness of
intertwining operator (up to scaling) implies 2).
\end{proof}

\begin{rk}
We were able to prove only one statement of the theorem
in the terms of
$\Sigma$ because of difficulties with an extension of $\phi$ to $\Sigma$.
\end{rk}

Now let us present some counterexamples to this statement
for pathological (monster) discrete groups.
Suppose, an infinite discrete group
$G$ has a finite number of conjugacy classes.
Such examples can be found in \cite{serrtrees} (HNN-group),
\cite[p.~471]{olsh} (Ivanov group), and \cite{Osin} (Osin group).
Then evidently, the characteristic function of unity element is not
almost-periodic and the argument above is not valid. Moreover, let us
show, that these groups give rise counterexamples to above theorem.

\begin{ex}\label{ex:osingroup}
For the Osin group the Reidemeister number $R(\Id)=2$,
while there is only trivial (1-dimensional) finite-dimensional
representation.
Indeed, Osin group is
an infinite finitely generated group $G$ with exactly two conjugacy classes.
All nontrivial elements of this group $G$ are conjugate. So, the group $G$
is simple, i.e. $G$ has no nontrivial normal subgroup.
This implies that group $G$ is not residually finite
(by definition of residually finite group). Hence,
it is not linear (by Mal'cev theorem \cite{malcev}, \cite[15.1.6]{Robinson})
and has no finite-dimensional irreducible unitary
representations with trivial kernel. Hence, by simplicity of $G$, it has no
finite-dimensional
irreducible unitary representation with nontrivial kernel, except of the
trivial one.

Let us remark that Osin group is non-amenable, contains the free
group in two generators $F_2$,
and has exponential growth.
\end{ex}

\begin{ex}\label{ex;ivanovgroup}
For large enough prime numbers $p$,
the first examples of finitely generated infinite periodic groups
with exactly $p$ conjugacy classes were constructed
by Ivanov as limits of hyperbolic groups (although hyperbolicity was not
used explicitly) (see \cite[Theorem 41.2]{olsh}).
Ivanov group $G$ is infinite periodic
2-generator  group, in contrast to the Osin group, which is torsion free.
The Ivanov group $G$ is also a simple group.
The proof (kindly explained to us by M. Sapir) is the following.
Denote by $a$ and $b$ the generators of $G$ described in
\cite[Theorem 41.2]{olsh}.
In the proof of Theorem 41.2 on  \cite{olsh}
it was shown that each of elements of $G$
is conjugate in $G$ to a power of generator $a$ of order $s$.
Let us consider any normal subgroup $N$ of $G$.
Suppose
$\gamma \in N$. Then $\gamma=g a^sg^{-1}$ for some $g\in G$ and some $s$.
Hence,
$a^s=g^{-1} \gamma g \in N$ and from  periodicity of $a$, it follows that also
$ a\in N$
as well as $ a^k \in N$  for any $k$, because $p$ is prime.
Then any element $h$ of $G$ also belongs to $N$
being of the form $h=\til h a^k (\til h)^{-1}$, for  some $k$, i.e., $N=G$.
Thus, the group $G$ is simple. The discussion can be completed
in the same way as in the case of Osin group.
\end{ex}

\begin{ex}
In paper \cite{HNN}, Theorem III and its corollary,
G.~Higman, B.~H.~Neumann, and H.~Neumann
proved that any locally infinite countable group $G$
can be embedded into a countable group $G^*$ in which all
elements except the unit element are conjugate to each other
(see also \cite{serrtrees}).
The discussion above related Osin group remains valid for $G^*$
groups.
\end{ex}

Let us remark that almost polycyclic group are residually finite
(see e.g. \cite[5.4.17]{Robinson})
while the groups from these counterexamples are not residually finite,
as it is clear by definition. That is why we would like to
complete this section with the following question.

\noindent{\bf Question.}
Suppose $G$ is a residually finite group and $\phi$ is its endomorphism
with finite $R(\phi)$. Does $R(\phi)$ equal $S_f(\phi)$?

The results of the present section are generalized to the case
of coincidences in \cite{bitwisted}.

\section{Baumslag-Solitar groups $B(1,n)$}\label{sec:BaumSolit}

In this section based on (a part of) \cite{FelGon} it is proved
that any injective endomorphism of a Baumslag-Solitar group
$B(1,n)$ has infinite Reidemeister number. In \cite{FelGon} the
similar result for automorphisms of $B(m,n)$ is obtained as well.

Let $B(1,n)=\langle a,b: a^{-1}ba=b^n, n>1 \rangle $ be
the Baumslag-Solitar groups. These groups (see \cite{fa-mo1})
are finitely-generated solvable groups (in particular amenable) which are not
virtually nilpotent. These  groups have
exponential growth \cite{harpe}  and  they are not Gromov hyperbolic.
Furthermore, these groups are  torsion  free and metabelian (an extension of
an Abelian group by an Abelian). More precisely one has

\begin{prop}\label{prop:felgon3.0}
$
B(1,n)\cong {\Z}[1/n]\rtimes_{\theta} {\Z},
$
where the
action of ${\Z}$ on ${\Z}[1/n]$ is given by $\theta(1)(x)=x/n$.
\end{prop}

\begin{proof}
The map defined by   $\iota(a)=(0,1)$ and
$\iota (b)=(1,0)$ extends to a  unique homomorphism
$\iota:B(1,n) \to \Z[1/n]\rtimes \Z,$
because
$$
\iota(a^{-1}) * \iota(b) * \iota(a)=
(0,-1) * (1,0) * (0,1)=(0,-1) * (1,1)=(n,0)=\iota(b^n).
$$
One has
\begin{equation}\label{ur:felgondob}
  \iota(a^r b^s a^{-r})=(0, r) * (s, -r)= \left(\frac
  s{n^r},0\right).
\end{equation}

The map $\iota$ is clearly surjective. Let us show
that this homomorphism is injective.
Let us remark that the group relation implies
$a^{-1}b^{-1}a=b^{-n}$. Hence for any $s$ one has
$a^{-1} b^s a=b^{ns}$.
Thus we can move all $a^{-1}$ to the right (until they
annihilate with some $a$ or take the extreme right place)
and all $a$ to the left, with an appropriate changing of powers
of $b$. Hence
any word in $B(1,n)$
is equivalent to a word of the form
$w=a^{r_1}b^{s}a^{r_2}$, where $r_1 \ge 0$, $r_2 \le 0$.
Then
$\iota(w)=(m,r_1 + r_2)$ for some $m\in \Z[1/n]$.
Hence, if $r_1+r_2\ne 0$, then $\iota(w)\ne e$. Let now
$r_1+r_2=0$, then by (\ref{ur:felgondob})
if $\iota(w)=e$, then $s=0$ and $w=e$.
\end{proof}

Consider the homomorphism $|\ \ |_a: B(1,n) \longrightarrow \Z $ which
associates to each word $w\in B(1,n)$ the sum of the exponents of $a$ in
this word. Since this sum for the relation is zero,
this is a well defined map, which
is evidently surjective.

\begin{prop}\label{prop:felgon3.1}
We have a short exact sequence
$$
0 \longrightarrow K\longrightarrow B(1,n)
\stackrel{|\ \ |_a}\longrightarrow \Z \longrightarrow 1,
$$
where $K$ is the kernel of $ |\ \ |_a$.
Moreover, $B(1,n)$ equals a semidirect product $K \rtimes \Z$.
\end{prop}

\begin{proof}
The first statement follows from surjectivity of $|\cdot|_a$.
Since $\Z$ is free, this sequence splits.
\end{proof}

\begin{prop}\label{prop:felgon3.2}
The kernel $K$ coincide with the
normalizer  $ N\langle b\rangle $ of the subgroup $\langle b \rangle$
generated by $b$ in $B(1,n)$:
\begin{equation}\label{ur:felgon3.2}
  0 \longrightarrow N\langle b\rangle \longrightarrow B(1,n)
\stackrel{|\ \ |_a}\longrightarrow \Z \longrightarrow 1.
\end{equation}
\end{prop}

\begin{proof}
We have $ N\langle b\rangle\subset K $. The quotient $B(1,n)/N\langle
b\rangle $ has the following presentation: $\ov a^{-1}\ov b \ov a=
\ov b^n$, $\ov b=1$. Therefore this group is isomorphic to ${\Z}$
 under the identification $[a]\leftrightarrow 1_\Z$.
Hence the
natural projection coincides with the map $|\ \ |_a$
and we have the following commutative diagram
$$
\xymatrix{
0 \ar[r] &  N\langle b\rangle \ar[r]\ar[d]^\cap & B(1,n) \ar[r]\ar[d]^= &
B(1,n)/(N\langle b\rangle)\ar[r]\ar[d]^\cong & 1  \\
0 \ar[r]& K \ar[r] & B(1,n) \ar[r] & {\Z} \ar[r] & 1.
}
$$
The five-lemma completes the proof.
\end{proof}

\begin{prop}\label{prop:felgon3.4}
Any homomorphism $\phi: B(1,n) \to B(1,n)$ is a homomorphism of
the short exact sequence {\rm (\ref{ur:felgon3.2})}.
\end{prop}

\begin{proof}
Let $\bar \phi$ be the  homomorphism induced by $\phi$ on
the abelianization $B(1,n)_{ab}$ of $B(1,n)$. The group
$B(1,n)_{ab}$ is isomorphic to $Z_{n-1} +Z$. The torsion elements of
$B(1,n)_{ab}$ form a subgroup isomorphic to $Z_{n-1}$ which is invariant
under any homomorphism. The preimage of this subgroup under the projection
$B(1,n) \to B(1,n)_{ab}$ is  exactly the subgroup
$N(b)$, i.e. the elements represented by words where the sum of the powers
of $a$ is zero. So it follows that $N(b)$ is mapped into $N(b)$ and the
result follows.
\end{proof}

\begin{teo}\label{teo:felgon3.4}
For any injective
homomorphism $\phi$ of $B(1,n)$ the  Reidemeister number  is
infinite.
\end{teo}

\begin{proof}
By Proposition \ref{prop:felgon3.4} it is a
homomorphism of short exact sequence. The induced map $\ov\phi$
on the quotient is
an injective endomorphism of $\Z$. If $\ov \phi=\Id_\Z$,
then by (\ref{eq:epiofclassforexs}) the number of Reidemeister
classes is  infinite. Hence, $\ov \phi$ is multiplication by $k\ne 0,1$.
But this is impossible. Indeed, when we apply
$\phi$ to the relation $a^{-1}ba=b^n$, under the identification of
Proposition \ref{prop:felgon3.0}
$\iota:B(1,n)\cong \Z[1/n]\rtimes \Z$  we have,
because $\phi(b)\in N\<b\>$ and hence $\iota(\phi(b))=(d,0)$ for
some $d\in\Z[1/n]$,
$$
(n d,0)=\iota(\phi (b^n))=
\iota(\phi(a^{-1}ba))=\iota(a^{-k} \phi(b) a^{k})=
(d\cdot n^k,0).
$$
This implies that either  $n^{1-k}=1$ or $\phi(b)=0$.
\end{proof}


\def\cprime{$'$} \def\cprime{$'$} \def\cprime{$'$} \def\cprime{$'$}
  \def\cprime{$'$} \def\dbar{\leavevmode\hbox to 0pt{\hskip.2ex
  \accent"16\hss}d} \def\cprime{$'$} \def\cprime{$'$}
  \def\polhk#1{\setbox0=\hbox{#1}{\ooalign{\hidewidth
  \lower1.5ex\hbox{`}\hidewidth\crcr\unhbox0}}}
\providecommand{\bysame}{\leavevmode\hbox to3em{\hrulefill}\thinspace}
\providecommand{\MR}{\relax\ifhmode\unskip\space\fi MR }
\providecommand{\MRhref}[2]{%
  \href{http://www.ams.org/mathscinet-getitem?mr=#1}{#2}
}
\providecommand{\href}[2]{#2}

\end{document}